\documentclass[journal,twoside,web]{ieeecolor}
\usepackage{generic}
\usepackage{cite}
\usepackage{amsmath,amssymb,amsfonts}
\usepackage{graphicx}
\usepackage{algorithm,algorithmic}
\usepackage{hyperref}
\usepackage{textcomp}
\usepackage{multirow}
\usepackage{tikz}
\usepackage{epsfig,verbatim}
\usepackage{mathtools}
\usepackage{makecell}
\usepackage[nolist]{acronym}

\usepackage[textsize=tiny]{todonotes}

\DeclareMathAlphabet{\mathcal}{OMS}{cmsy}{m}{n}
\DeclareSymbolFont{largesymbols}{OMX}{cmex}{m}{n}
\let\sum\relax
\DeclareSymbolFont{CMlargesymbols}{OMX}{cmex}{m}{n}
\DeclareMathSymbol{\sum}{\mathop}{CMlargesymbols}{"50}
\DeclarePairedDelimiter{\norm}{\lVert}{\rVert}
\DeclareMathOperator{\dom}{dom}
\DeclareMathOperator{\NC}{\mathsf{NC}}
\DeclareMathOperator{\Id}{Id}
\DeclarePairedDelimiter{\inner}{\langle}{\rangle}
\DeclareMathOperator{\prox}{prox}
\DeclareMathOperator{\diag}{diag}
\DeclareMathOperator{\opV}{\mathbb{V}}
\DeclareMathOperator{\opT}{\mathbb{T}}
\DeclarePairedDelimiter{\abs}{\lvert}{\rvert}

\newtheorem{theorem}{Theorem}
\newtheorem{definition}{Definition}
\newtheorem{proposition}{Proposition}
\newtheorem{lemma}{Lemma}

\newtheorem{remark}{Remark}
\newtheorem{assumption}{Assumption}

\usetikzlibrary{calc,shapes.geometric,positioning,arrows,intersections}

\renewcommand{\baselinestretch}{1}

\newcommand{\bigO}{\mathcal{O}}
\newcommand{\R}{\mathbb{R}}
\newcommand{\eps}{\epsilon}
\newcommand{\x}{{\bf x}}

\newcommand{\vv}{{\bf v}}

\newcommand{\z}{{\bf z}}

\newcommand{\op}{\operatorname}
\newcommand{\Zer}{\mathsf{Zer}}
\newcommand{\zer}{\mathsf{Zer}}

\newcommand{\EE}{{\mathbb{E}}}
\newcommand{\Ex}{\mathbb{E}}
\newcommand{\PP}{{\mathbb{P}}}

\newcommand{\mc}{\mathcal}

\newcommand{\0}{\mathbf{0}}
\newcommand{\1}{\mathbf{1}}
\newcommand{\scrA}{\mathcal{A}}

\newcommand{\setC}{\mathsf{C}}
\newcommand{\setD}{\mathsf{D}}

\newcommand{\scrF}{\mathcal{F}}

\newcommand{\setH}{\mathsf{H}}
\newcommand{\scrI}{\mathcal{I}}

\newcommand{\setM}{\mathsf{M}}
\newcommand{\scrN}{\mathcal{N}}

\newcommand{\scrQ}{\mathcal{Q}}

\newcommand{\setX}{\mathsf{X}}

\newcommand{\setOmega}{\mathsf\xi}

\begin{acronym}
\acro{GNEP}[GNEP]{Generalized Nash equilibrium problem}
\acro{GNE}[GNE]{Generalized Nash equilibrium}
\acro{VE}[VE]{Variational equilibrium}
\acro{VI}[VI]{Variational inequality}
\acro{SGNEP}[SGNEP]{stochastic generalized Nash equilibrium problem}
\end{acronym}
\newcommand{\bA}{{\mathbf A}}

\newcommand{\bL}{\mathbf{L}}
\newcommand{\bI}{\mathbf{I}}

\newcommand{\bD}{\mathbf{D}}

\newcommand{\bW}{\mathbf{W}}

\newcommand{\xbold}{{\bf x}}

\newcommand{\ubold}{{\bf u}}

\def\BibTeX{{\rm B\kern-.05em{\sc i\kern-.025em b}\kern-.08em
    T\kern-.1667em\lower.7ex\hbox{E}\kern-.125emX}}
\begin{document}
\title{Complexity guarantees for risk-neutral generalized Nash equilibrium problems}
\author{Haochen Tao, Andrea Iannelli, Meggie Marschner, Mathias Staudigl, Uday V. Shanbhag, and Shisheng Cui 
\thanks{Shisheng Cui acknowledges supported in part by NSFC under Grant 62373050 while Uday Shanbhag acknowledges partial support from ONR Grant N00014-22-1-2589, AFOSR Grant FA9550-24-1-0259, and DOE Grant DE-
SC0023303. This research benefited from the support of the FMJH Program Gaspard Monge for optimization and operations research and their interactions with data science. }
\thanks{H. Tao and S. Cui are with the School of Automation, Beijing Institute of Technology, Beijing, China (e-mail: taohaochen,css@bit.edu.cn). }
\thanks{A. Iannelli is with the Institute for Systems Theory and Automatic Control (IST), University of Stuttgart (e-mail: andrea.iannelli@ist.uni-stuttgart.de).}
\thanks{M. Marschner and M. Staudigl are with the Department of Mathematics, University of Mannheim (e-mail: meggie.marschner,m.staudigl@uni-mannheim.de).}
\thanks{U.V. Shanbhag is with the Department of Industrial and Operations Engineering
University of Michigan Ann Arbor, MI (e-mail: udaybag@umich.edu).}}

\maketitle
\begin{abstract}
In this paper, we address \ac{SGNEP} seeking with risk-neutral agents. Our main contribution lies the development of a stochastic variance-reduced gradient (SVRG) technique, modified to contend with general sample spaces, within a stochastic forward-backward-forward splitting scheme for resolving structured monotone inclusion problems. This stochastic scheme is a double-loop method, in which the mini-batch gradient estimator is computed periodically in the outer loop, while only cheap sampling is required in a frequently activated inner loop, thus achieving significant speed-ups when sampling costs cannot be overlooked. The algorithm is fully distributed and it guarantees almost sure convergence under appropriate batch size and strong monotonicity assumptions. Moreover, it exhibits a linear rate with possible biased estimators, which is rather mild and imposed in many simulation-based optimization schemes. Under monotone regimes, the expectation of the gap function of an averaged iterate diminishes at a suitable sublinear rate while the sample-complexity of computing an $\epsilon$-solution is provably $\mathcal{O}(\epsilon^{-3})$. A numerical study on a class of networked Cournot games reflects the performance of our proposed algorithm.
\end{abstract}

\begin{IEEEkeywords}
Stochastic Generalized Nash equilibrium, Distributed Control, Multi-Agent systems, Networks 
\end{IEEEkeywords}

\section{Introduction}
In \acp{GNEP}, the goal of each agent is to minimize its individual cost function, while taking global coupling constraints into account. The topic has attracted much attention in the control community, particularly in multi-agent network systems \cite{ardagna2015generalized,10237320}. The qualifier ``Generalized'' means that both the local objective function and the feasible set of any player is influenced by the decisions made by its rivals~ \cite{FacKan07,facchinei2010,FacPalPanScu10}. In data-driven settings, we aim to compute equilibria of \acp{GNEP} in which both the local cost functions and/or the constraints depend on random noise terms. The presence of uncertainty makes the \ac{GNEP} even more challenging and successfully mastering these types of equilibrium problems is still a topic of intensive research. 

A major innovation within computational game theory over the last 10-20 years has been the development of scalable numerical algorithms for iteratively computing a variational equilibrium. A successful algorithmic design should exploit the distributed character of the game problem and thus integrate as much parallelization as possible. Particularly important milestones in this direction have been the early work in \cite{facchinei2014vi} and the operator splitting approach of \cite{pavel2019}. These papers provided the starting point for a rich literature, exploring the effectiveness of monotone operator theory \cite{BauCom16} as a toolbox for developing distributed control schemes in large-scale networked systems.

In data-driven control settings, uncertainty cannot be overlooked~\cite{fabiani2022stochastic,MerStaIFAC19,yu2017}, and thus, \acp{SGNEP} arise. For example, \acp{SGNEP} are utilized in electricity dispatch scenarios \cite{henrion2007} and in limited natural gas markets \cite{abada2013}. In such cases, the networked Cournot game model, which describes the situation in which firms compete in limited markets, is commonly used. Randomness in these problems usually arises from uncertain demand \cite{abada2013,demiguel2009}. Other examples include problems in shape optimization with multiple objectives~\cite{tang2007multicriterion} and PDE-constrained optimization~\cite{ramos2002pointwise,gahururu2023risk}. 

Distributed computing is another challenge. Owing to the structure or scale of the network, communication is more or less restricted between the participants. In the above cases, players can have access information, relevant to themselves and several trusted neighbors. That is to say, each player has to make independent decisions without the help of a central operator. Previous research on this topic has frequently framed the problem in terms of a \ac{VI}, which can be formulated as a structured monotone inclusion~\cite{cui2023variance,Bot2021}. The latter can be addressed through distributed iterations by applying operator splitting methods~\cite{BauCom16,YiPav19,belgioioso2018}. The resulting solution sequence converges to a subset of \ac{GNE}, known as \ac{VE}, which have a clear economic interpretation \cite{kulkarni2012}.

The stochastic nature specific to SGNEP poses challenges for the computation of the gradient mapping. Specifically, we assume that expectation  cannot be evaluated, a consequence of either the computational challenge in evaluating a complicated multi-dimensional integral or because the distribution of the noise contained within the model is unknown a priori. Consequently, the objective function cannot be directly computed as in deterministic cases. Thus, stochastic approximation (SA) theory is introduced to estimate the truth value of the gradient mapping by sampling realizations of the mapping under perturbations. Usually, variance-reduced sampling methods are utilized simultaneously to continuously improve the approximation accuracy \cite{Bot2021,IusJofOliTho17,alacaoglu2022stochastic}.

We briefly review the state-of-the-art in the computation of \ac{GNE}. Yi \cite{YiPav19} proposes a distributed preconditioned forward-backward algorithm in a deterministic environment, assuming that the mapping involved is strongly monotone and Lipschitz continuous. Studies such as \cite{franci2021fbtac,JiaxXu09,rosasco2016} and \cite{franci2020fbecc} focus on the stochastic variant, among which \cite{rosasco2016} achieves a convergence rate of $\bigO(1/T)$ under the strongly monotone assumption. However, many studies on \acp{SGNEP} lack a rate statement. To enhance the performance and achieve deterministic convergence rates, variance-reduction techniques are introduced to improve the accuracy of expectation-valued map approximation using mini-batch stochastic approximation schemes. This approach has proven highly effective for both smooth \cite{jofre2019variance} and non-smooth \cite{jalilzadeh2018smoothed} convex/nonconvex stochastic optimization and variance-reduced quasi-Newton schemes \cite{jalilzadeh20variable}.

To achieve further acceleration, especially when sampling is no longer cost-effective, various techniques such as stochastic variance reduced gradient (SVRG) \cite{johnson2013accelerating,nan2023extragradient,alacaoglu2022stochastic} and stochastic average gradient algorithm (SAGA) \cite{NIPS2014_ede7e2b6} are implemented. Unlike mini-batch schemes, which increase the batch sizes in each iteration, the entire set of samples is periodically employed  in SVRG. Generally, SVRG is utilized to take advantage of the finite-sum structure found in optimization and variational inequality problems commonly encountered in machine learning and engineering. The original SVRG algorithm \cite{johnson2013accelerating} is incorporated within a double loop structure and adapted to the typical finite-sum structure found in empirical risk minimization, which has been extended to saddle-point problems \cite{palaniappan2016stochastic} and monotone variational inequalities \cite{nan2023extragradient,alacaoglu2022stochastic,zhang2023communication}, as shown in Table 1. Among those works on VI, \cite{zhang2023communication,alacaoglu2021forward,cai2022stochastic,huang2022accelerated} demonstrate linear convergence when strong monotonicity is presented. \cite{nan2023extragradient} incorporates variance reduction in extragradient (EG) algorithms and also produces linear rate. It is worth noting that \cite{nan2023extragradient} relies on an error-bound condition, which is weaker than strong monotonicity but is still strict in practice. The strong monotonicity assumption is further relaxed in \cite{alacaoglu2022stochastic} and \cite{huang2022accelerated}, where merely monotone operators are required but the convergence rate is reduced to $\bigO(1/T)$.

Several research gaps arise when we go through the prior work. We typically discover that most state-of-art SVRG schemes ($i$) do not allow for general probability spaces, ($ii$) require conditionally unbiased oracles, and ($iii$) are tailored for a specific type of problems such as convex-concave saddle-point problems and monotone VIs. In light of this, we develop amongst the first distributed variance-reduced SFBF algorithm (DVRSFBF) applicable to structured monotone inclusion problems, which subsume monotone VI, convex-concave saddle-point problems, and convex optimization problems, etc. The main results of our study ensure linear convergence to a variational equilibrium under strong monotonicity and a relatively mild assumption of asymptotically unbiased oracles. Further, under a milder monotonicity assumption, the sequence admits a sublinear rate with a sample-complexity of $\mathcal{O}(\epsilon^{-3})$ for computing an $\epsilon$-solution. At each iteration, average gradients are computed in the outer loop while cheap samples are used in the frequently activated inner loop, leading to improved performance and reduced sampling costs.
This is a major improvement compared to the existing distributed equilibrium seeking algorithms in the stochastic regime that are based on the forward-backward splitting scheme. To the best of our knowledge, no prior variance-reduced algorithm in the same setting has provided convergence rate guarantees for general monotone problems.

\noindent\textbf{Contributions:} The proposed VR splitting scheme
\begin{itemize}
	\item  is fully distributed for resolving GNEP and converges to a variational equilibrium in an a.s. sense;
	\item equips with variance reduction which may contend with possible infinite sample space;
	\item allows for possible biased estimator with optimal linear convergence rate in strongly monotone settings.
        \item can accommodate merely monotone maps, providing sublinear rate and sample-complexity guarantees.
\end{itemize}

\begin{table}[t]
	\caption{SVRG-based VI literature review}
	\label{tab.SVRG}
	\begin{center}	
		\renewcommand\arraystretch{1.4}
		\begin{tabular}{|c | c | c | c | c|}
			\hline
			Paper & \makecell{Finite \\ sum}  & Mono. & Biased & Statements \\
			\hline
			\cite{alacaoglu2022stochastic} & Y & Mere & N & $\mathbb{E}[\mbox{Gap}(\x_T)] \le \bigO(1/T)$ \\ 
			\hline
			\cite{alacaoglu2021forward} & Y & Strong & N & $\mathbb{E}[\|\x^T-\x^*\|^2] \le \bigO(q^T)$ \\ 
			\hline
			\cite{cai2022stochastic} & N & Strong & N & $\mathbb{E}[\|\x^T-\x^*\|^2] \le \bigO(q^T)$ \\
			\hline
			\cite{nan2023extragradient} & Y & \makecell{Mere, \\ Error bound} & N & $\mathbb{E}[\|\x^T-\x^*\|^2] \le \bigO(q^T)$ \\
			\hline
			\cite{zhang2023communication} & Y & Strong & N & $\mathbb{E}[\|\x^T-\x^*\|^2] \le \bigO(q^T)$ \\
			\hline
			\multirow{2}{*}{\cite{huang2022accelerated}} & \multirow{2}{*}{Y} & Strong & N & $\mathbb{E}[\|\x^T-\x^*\|^2] \le \bigO(q^T)$ \\
			\cline{3-5}&& Mere & N & $\mathbb{E}[\|\x^T-\x^*\|^2] \le \bigO(1/T)$ \\
			\hline
			\multirow{2}{*}{This work} & \multirow{2}{*}{N} & Strong & Y & $\mathbb{E}[\|\x^T-\x^*\|^2] \le \bigO(q^T)$ \\
			\cline{3-5}&& Mere & N & $\mathbb{E}[\mbox{Gap}(\x_T)] \le \bigO(1/T)$ \\
			\hline
		\end{tabular}
	\end{center}
\end{table}

\subsection{Preliminaries}
$\R$ represents the set of real numbers and $\bar\R=\R\cup\{+\infty\}$. The symmetric bilinearform $\langle\cdot,\cdot\rangle:\R^n\times\R^n\to\R$ indicates the standard inner product and $\norm{\cdot}$ represents the associated Euclidean norm on $\R^{n}$.
A symmetrix matrix $\Phi$ is positive definite, denoted as $\Phi\succ0$, if $x^\top \Phi x>0$. For $\Phi\succ0$, $\langle x, y\rangle_{\Phi}\triangleq\langle \Phi x, y\rangle$ represents the scalar product induced by $\Phi$ and its norm is $\norm{x}_{\Phi}=\sqrt{\langle \Phi x, x\rangle}$. We indicate the Kronecker product between matrices $C$ and $D$ by $C\otimes D$. ${\bf{0}}_n$ (${\bf{1}}_n$) is the n-dimensional vector where every element is $0$ ($1$). For $x_{1}, \ldots, x_{N} \in \R^{n}$, $\boldsymbol{x} :=\op{col}\left(x_{1}, \dots, x_{N}\right)=\left[x_{1}^{\top}, \dots, x_{N}^{\top}\right]^{\top}.$

$\mit\Gamma:\R^{n}\rightrightarrows \R^n$ is defined as a set-valued operator: $\dom({\mit\Gamma})= \{x\in\R^n\mid {\mit\Gamma}(x)\neq\emptyset\}$ denotes the domain of $\mit\Gamma$, and $\Zer(\mit\Gamma)= \{x\in\R^n\mid \text{0} \in \mit\Gamma(x)\}$ represents the set of zeros of $\mit\Gamma$. 
The resolvent of the operator $\mit\Gamma$ is defined as $\mathrm{J}_{\mit\Gamma}= (\Id+ \mit\Gamma)^{-1}$, where $\op{Id}$ indicates the identity operator. 
A monotone operator $\mit\Gamma$ satisfies that $\inner{\mit\Gamma(x)-\mit\Gamma(y),x-y}\geq$ 0. Lipschitz continuity is defined that for some $l>0$, $\|\mit\Gamma(x)-\mit\Gamma(y)\| \leq l\|x-y\|$ for all $x,y\in\dom(\mit\Gamma)$.
For a proper, convex, and lower semi-continuous function $h$, we denote its subdifferential by the operator $\partial h(x):=\{w\in\dom(h) \mid \langle y-x,w\rangle\leq h(y)-h(x),\forall y\in\dom(h)\}$. $\prox_{h}(z):=\op{argmin}_{w\in\setOmega}\{h(w)+\tfrac{1}{2}\norm{w-z}^{2}_{}\}=\mathrm{J}_{\partial h}(z)$ indicates the proximal operator. We define
the indicator function of the set $Q$ as $\iota_Q$, where $\iota_Q(x)=1$ if $x\in Q$ and $\iota_Q(x)=0$ otherwise. Let $\op{proj}_{\setX}(y) = \arg\min_{x\in\setX}\norm{y-x}_2^2$ represent the projection of $x$ onto $\setX$. The set-valued mapping $\mathrm{N}_{Q} : \R^{n} \rightrightarrows \R^{n}$ denotes the normal cone operator for the set $Q$, i.e., $\mathrm{N}_{Q}(x)=\varnothing$ if $x \notin Q,\left\{u \in \R^{n} | \sup _{y \in Q} u^{\top}(y-x) \leq 0\right\}$ otherwise.

\subsection{Motivating examples}
\subsubsection{Electricity dispatch with uncertain demand}
One application of SGNEPs is the electricity dispatch problem without knowing the true demand \cite{henrion2007,escobar2005oligopolistic,hu2004electricity,Kannan2013}. Here we present a simplified version of the model for competition in electricity spot markets. There are $N$ nodes where electricity is generated and each has its associated demand. Let $\mc I := \{1,2,\dots,N\}$. The nodal generation and demand at node $i$ are denoted by $q_i$, $d_i$ respectively.

The transmission network of the market is modeled by a directed graph, with its $M$ edges representing transmission lines. In this example, transmission losses are neglected. $y_m$ indicates the transmission power on the $m$-th edge. The incidence matrix of the network is $B\in \R^{N\times m}$, where $B_{im}$ equals to 1 if the $i$-th node is connected by the $m$-th edge, otherwise $B_{im} = 0$. To meet the electricity demand at each node, the following condition must be satisfied:
\begin{equation}\label{eq:power_demand}
	q + By \geq d,
\end{equation}
where $q := \text{col}(q_i)_{i=1}^N \in \R^N$, $d := \text{col}(d_i)_{i=1}^N \in \R^N$ and $y := \text{col}(y_m)_{m=1}^M \in \R^M$. Condition \eqref{eq:power_demand} plays the role of global constraints in generalized Nash game. Typically, $q$ and $y$ are bounded as follows,
\begin{equation}\label{eq:power_limit}
	0\leq q\leq \bar{q}, \quad \|y\| \leq \bar{y},
\end{equation}
where $\bar{q}$ and $\bar{y}$ are the upper limits of generation capacity and transmission load. The inequality signs in \eqref{eq:power_limit} are to be understood component-wise. 

The $i$-th generator bids to an independent system operator (ISO), with its cost function given by:
\begin{equation}\label{eq:power_cost}
	c_i(q_i) = \gamma_i q_i^2 + \tau_i q_i, \quad \forall i \in \mc I.
\end{equation}
This is a setup similar to \cite{hu2004electricity}, assuming that the quadratic term coefficient $\gamma_i$ in the cost function is positive for all $i \in \mc I$ to ensure convexity. Other types of cost functions can be found in \cite{escobar2005oligopolistic} and \cite{Kannan2013}. We use the quadratic example just for illustration. 

The ISO selects a power purchasing strategy that minimizes total costs while satisfying the aforementioned constraints. Because the cost functions are completely separable in the control variables, the problem can be decomposed and described as a collection of individual minimization problems, subject to local and global coupling constraints:
\begin{equation}
	\forall i\in\mc I: \quad\left\{\begin{array}{cl}
		\min\limits_{p,y} & c_i(q_i) \\ 
		\text { s.t. } & (q,y)\in G,
	\end{array}\right.
\end{equation}
where
\begin{align}
	\notag G = \{&(q,y)\in \R^{N+m}: q+By\geq d,\\
	& 0\leq q\leq \bar{q}, \|y\| \leq \bar{y}\}
\end{align}

However, in practice, each participant $i$ has no access to the actual demand of other nodes or even the $i$-th node. Typically, players can only obtain estimated values by some approaching schemes based on historical data \cite{henrion2007,Kannan2013}. Therefore, the demand $d$ in constraint \eqref{eq:power_demand} is a random vector defined on the probability space $(\xi, \mc F, \PP)$, with an unknown distribution. As a result, the constraint set $G$ accordingly becomes: 
\begin{align}
	\notag \tilde{G}(\xi) = \{&(q,y)\in \R^{N+m}: q+By\geq d(\xi),\\
	& 0\leq q\leq \hat{q}, \|y\| \leq \bar{y}\}
\end{align}
where $\hat{q}$ denotes the upper bound on the generation vector $q$. Now $p$ and $y$ within the feasible domain are no longer deterministic but depend on the random variable $\xi$. Meanwhile, the objective function can be expressed as the corresponding expectation-valued version. Thus, the original problem is transformed into a SGNEP defined as:
\begin{equation}\label{eq:power_problem}
	\forall i\in\mc I: \quad\left\{\begin{array}{cl}
		\min\limits_{p,y} & \EE\left[c_i(q_i(\xi))\right] \\ 
		\text { s.t. } & (q(\xi),y(\xi))\in G(\xi), \PP\text{-a.s.}.
	\end{array}\right.
\end{equation}

Due to the uncertainty of electricity demand, problem \eqref{eq:power_problem} becomes significantly more complex. However, such scheduling problems with uncertain demands can be modelled as SGNEPs, which provides an efficient way to addressing them by designing appropriate Nash equilibrium seeking schemes.

\subsubsection{Ride-hailing platforms competition}
A subclass of ride-hailing platforms competition problems can be modelled as SGNEPs \cite{ZHONG2022,offplatform23,ridesharing19,ride-hail22}. In this type of problem, multiple ride-hailing platforms compete for an unknown number of potential customers within several markets or areas in order to maximize their own profits. Consider $N$ ride-hailing firms provide services in $H$ areas, where $\mc I := \{1,2,\dots,N\}$ and $\setH := \{1,2,\dots,H\}$. By designing pricing strategies $p_i = \text{col}(p_{i,h})_{h=1}^H$ and the wages paid to registered drivers $w_i = \text{col}(w_{i,h})_{h=1}^H$, platform $i$ seeks to maximize its total profit. Besides, the average price in area $h$ should be limited by a ceiling $\bar{p}_h >0$ set by consumers' associations, i,e., 
\begin{equation}\label{price}
	\forall h\in \setH, \, \frac{1}{N}\sum_{i\in \mc I}p_{i,h}<\bar{p}_h.
\end{equation}
We notice that the above joint constraints involve the entirety of participants. Firm $i$'s strategy set depends on the decisions made by opponents, which makes the optimization problem a generalized Nash equilibrium problem. 

The wage $w_{i,h}$ has a bottom line of $\underline{w}$, which is regulated by the government \cite{ZHONG2022}. Suppose $C_h$ denotes the potential riders in area $h$, among which the fraction of customers preferring platform $i$ is defined by the demand function:
\begin{equation}\label{demand}
	d_{i,h} = \frac{C_hK_{i,h}}{\bar{p}\sum_{j\in \mc I}K_{j,h}}\left(\bar{p}-p_{i,h}+\frac{\theta_i}{N-1}\sum_{j\in \mc I\backslash \{i\}}p_{j,h}\right),
\end{equation}
where $K_{i,h}$ is the number of registered drivers of firm $i$ in area $h$, $\bar{p}>\max_{h\in\setH}\bar{p}_h$ represents the maximum service price among all the areas, and $\theta_i\in [0,1]$ models the substitutability of the service provided by each firm. However, the demand request in \eqref{demand} does not account for the actual services provided by drivers. There's a trade-off between the pay-back and opportunity cost drivers pay. A random variable $\delta_{i,h}$ is therefore introduced to model the threshold. If the return $w_{i,h}$ reaches $\delta_{i,h}$, the transaction will be made. Thus the effective demand can be represented by 
\begin{equation}\label{edemand}
	d^e_{i,h} = d_{i,h}\PP[w_{i,h}\geq\delta_{i,h}].
\end{equation} 
Similarly, the fraction of drivers giving services is
\begin{equation}\label{drivers}
	k_{i,h} = K_{i,h}\PP[w_{i,h}\geq\delta_{i,h}].
\end{equation}
We can see that by raising salary levels, it helps to attract more drivers to actively provide services, which may in turn facilitate more orders, and thereby increase platform revenue.

When computing the profit of firm $i$, the sum of drivers in area $h$ is required, which is usually unknown to any individual platform. One practical method is replacing ${C_h}/(\sum_{j\in \mc I}K_{j,h})$ with $C_h(\xi)$ to explicitly deal with the uncertainty of the game \cite{ride-hail22}. To simplify the analysis, many existing works \cite{ZHONG2022,ridesharing19,ride-hail22} adopt a uniform distribution for the driver threshold variable $\delta_{i,h}$, which effectively reduces the uncertainty to a single source. The optimization problem can be described as:
\begin{equation}\label{opt_ride}
	\forall i\in\mc I: \, \left\{\begin{array}{cl}
		\max\limits_{p_i,w_i} & \EE_\xi\left[\sum_{h\in\setH}(p_{i,h}d^e_{i,h}(\xi)-w_{i,h}k_{i,h})\right] \\ 
		\text { s.t. } & \eqref{price},\eqref{demand},\eqref{edemand},\eqref{drivers},p_{i,h}\geq 0, \\
		& w_{i,h}\geq\underline{w}, \forall h \in \setH.
	\end{array}\right.
\end{equation}
The objective function in \eqref{opt_ride} is the profit from serving riders minus the wages paid to registered drivers. Due to the unknown uncertainty of potential customers, the cost is expectation-valued. At a time when ride-hailing services are in full swing, the demand for platforms to solve such game problems individually is gradually increasing. From the above, we notice that they can be well interpreted into a subclass of stochastic generalized Nash equilibrium seeking problems.

We close by underlining that stochastic generalized Nash equilibrium problems are prevalent in the field of multi-agents system control and decision-making, motivating us to develop feasible and efficient distributed methods to solve them.

\section{Stochastic Generalized Nash Equilibrium Problems}
\label{sec:prelims}
\subsection{Game Formulation}
We first define the generalized Nash equilibrium problem. Consider a non-cooperative game with a set of $N$ agents $\mc I=\{1,\ldots,N\}$. For an in-depth treatment of this class of problems, we refer to \cite{FacKan07} and \cite{FacPalPanScu10}. 
The decision made by player $i$ is $u_{i}\in\R^{d_i}$. Let $\ubold = \op{col}(\left\{u_i\right\}_{i\in \mc I}) \in \R^d$ denote the strategy profile where $d = \sum_{i\in \mc I}d_i$. Every player has a local cost $J_{i}: \R^{d} \to \bar\R$ which follows the form of
\begin{equation}\label{eq:theta}
	J_{i}(u_{i}, \ubold_{-i}):= f_{i}(u_{i},\ubold_{-i}) + g_{i}(u_{i}),
\end{equation}
where the vector $\ubold_{-i}=\op{col}(\left\{u_j\right\}_{j\in \mc I - \left\{i\right\}}) \in \R^{d-d_i}$ denotes the strategy profile except $u_i$. In \eqref{eq:theta}, $f_{i}(u_{i},\ubold_{-i})$ represents the smooth part of the cost function while $g_{i}:\R^{d_i}\to\R\cup\{\infty\}$ denotes the non-smooth part, such as indicator functions of local feasible sets, penalty functions or other non-smooth structure in specific problems. Moreover, $f_{i}(u_{i},\ubold_{-i})$ explicitly relies on both $u_i$ and a subset of opponents' control variables $\{u_{j}\}_{j\in\scrN_{i}^{A}}$, where $\scrN_{i}^{A}\subset\mc I$ is called the \emph{interaction neighborhood} of agent $i$. In this work, we impose the following assumptions on $J_{i}(u_{i},\ubold_{-i})$ as much GNEP literature does\cite{kulkarni2012,facchineikanzow2007,FacFisPic07}.
\begin{assumption}\label{ass:convex} \em
	For $\forall i \in\mc I$,
	
	(i) $f_{i}(\cdot,\ubold_{-i})$ in \eqref{eq:theta} is convex and continuously differentiable for all $\ubold_{-i}$;
	
	(ii) $g_i(\cdot)$ in \eqref{eq:theta} is convex and lower semi-continuous, and $\dom(g_{i})=\setD_i\subseteq\R^{d_i}$ is nonempty, compact, convex and bounded.
\end{assumption}

To maintain the joint convexity, coupling constraints are supposed to be affine in this work. Considering local constraints, the global feasible set can be defined as
\begin{equation}\label{eq:coupling}
	\setC := \{\ubold\in\setD\mid\; A\ubold-b \leq \0_{m}\},
\end{equation}
where $\setD:=\prod_{i=1}^N\setD_i$, $A:= [A_1\mid\dots\mid A_N]\in\R^{m\times d}$ and $b:= \sum_{i=1}^{N}b_{i}\in\R^m$. 
$A_i\in\R^{m\times d_i}$ in $A$ is the block related to player $i$. Given $\ubold_{-i}$, player $i$'s feasible set is the intersection of local and joint constraints denoted as
\begin{equation}\label{eq:coupling_i}
	\setC_{i}(\ubold_{-i}) = \{u_i \in \setD_i \mid A_i u_i -b_{i}\leq \sum_{j \neq i}^{N}(b_{j}-A_j u_j)\}.
\end{equation}
We note that while many results of this paper can be extended to certain classes of nonlinear constraints, we adhere to the linear formulation as it is the most important setting in distributed control and \acp{GNEP} \cite{belgioioso2023,franci2021fbtac,franci2022partial,YiPav19}.

Then player $i$'s optimization problem can be formulated for each player $i$ as
\begin{equation}\label{eq:BR}
	\forall i\in\mc I: \quad\left\{\begin{array}{cl}
		\min\limits_{u_i \in\R^{d_{i}}} & J_i(u_i, \ubold_{-i}) \\ 
		\text { s.t. } & u_{i}\in\setC_{i}(\ubold_{-i}).
	\end{array}\right.
\end{equation}
If problem \eqref{eq:BR} is solved for every participant simultaneously then we obtain the so-called \ac{GNE} \cite{facchinei2007vi,FacPan03}, an $N$-tuple $\ubold^{\ast}=\textnormal{col}(u_{1}^{\ast},\ldots,u_{N}^{\ast})\in\setD$ such that none of the players can benefit by unilaterally changing its decision, or
\begin{equation}\label{eq:game}
	J_i(u_{i}^{\ast},\ubold^{\ast}_{-i}) \leq \inf\{ J_i(u_{i},\ubold^{\ast}_{-i}) \, \mid  \, u_{i}\in\setC_i(\ubold^{\ast}_{-i})\}, \forall i \in{\mc I}.
\end{equation}

To ensure the existence of a \ac{GNE}, we assume Slater's constraint qualification to hold.
\begin{assumption}\label{ass:X} \em
	$\setC_{i}(\ubold_{-i})$ has a non-empty relative interior for $\forall i\in \mc I$ and $\ubold_{-i}.$ 
\end{assumption}
Our goal is to describe distributed stochastic algorithms which iteratively compute a special class of \ac{GNE} called \ac{VE}. These specific equilibrium points can be characterized as solutions to a class of constrained \acp{VI}. For this purpose we define the \emph{pseudogradient} as 
\begin{equation}\label{eq:F}
	F(\ubold):= \mathrm{col}\left( \nabla_{u_1}f_{1}(\ubold),\ldots,\nabla_{u_N}f_{N}(\ubold)\right)\in \R^d,
\end{equation}
and the associated variational inequality problem
\begin{equation}\label{eq:VI}
\sum_{i\in \mc I}(g_{i}(u_i)-g_{i}(u_i^*)) + \inner{F(\ubold^*),\ubold-\ubold^*} \ge 0,\, \forall \ubold \in \setC.
\end{equation}

Let Assumption \ref{ass:convex} and \ref{ass:X} hold, the solution set of VI($\setC,F$) \eqref{eq:VI} is nonempty \cite[Corollary 2.2.5]{FacPan03}.

\begin{assumption}
	\label{ass:GM} \em
	The pseudogradient mapping \eqref{eq:F} is monotone and $\ell$-Lipschitz continuous on $\setD$. 
\end{assumption}

Under this monotonicity assumption, we can use the Karush-Kuhn-Tucker (KKT) theorem for convex optimization problems, to derive a comprehensive set of necessary and sufficient conditions, characterization a solution to \eqref{eq:VI}. Indeed, it is well known \cite[Theorem 3.1]{facchinei2007vi}, \cite[Theorem 3.1]{auslender2000} that a solution of \eqref{eq:VI} is a solution of the following convex problem 
\begin{align*}
\min_{\ubold}\inner{F(\ubold^{*}),\ubold}+\sum_{i\in\mc I}g_{i}(u_{i})\quad \text{s.t.: } \ubold\in \setC
\end{align*}
Using Lagrangian techniques, it can be readily seen that the optimality conditions to this convex problem read as 
\begin{equation}\label{KKT_VI}
	\forall i\in\mc I:\begin{cases}
		\0_{d_i}\in \nabla_{u_{i}}f_{i}(u^{\ast}_{i},\ubold^{\ast}_{-i})+\partial g_{i}(u^{*}_{i})+A_{i}^{\top}y^{\ast}\\
		\0_{m}\in \op{N}_{\R^{m}_{\geq0}}(y^{\ast})-(A\ubold^{\ast}-b),\\
	\end{cases}
\end{equation}
Writing these player-specific conditions in compact vector notation, we arrive at the monotone inclusion 
$$
\0 \in \left[\begin{array}{cc} G(\ubold^*)\\ \NC_{\R^{m}_{\geq 0}}(y^{*})\end{array}\right]+\left[\begin{array}{cc}F(\ubold^{\ast})\\ b\end{array}\right]+\left[\begin{array}{cc} A^{\top}y^\ast \\ -A\ubold^{\ast}\end{array}\right]. 
$$
where $G:\R^{d}\rightrightarrows\R^{d}$ is the maximally monotone operator, defined as $G(\ubold):=(\partial g_{1}\times\cdots\times\partial g_{N})(\ubold)$.

\subsection{Stochastic Game Setup}\label{ssec:stochastic}
In stochastic generalized Nash game problems (SGNEPs) the local cost functions of the agents are affected by random elements so that they cannot be computed a-priori. Our model of uncertainty is defined on a common probability space $(\Omega,\scrF,\PP)$, with random variables $\xi^{i}:\Omega\to \Xi_{i}$, taking values in some measurable space $(\Xi_{i},\scrA_{i})$. Thus the local cost function of player $i$ turns into
\begin{equation}\label{eq:Ef}
	J_i(u_i,\ubold_{-i})=\Ex[\tilde{f}_{i}(\ubold,\xi^{i})]+g_i(u_i),
\end{equation}
where uncertainty is only supposed to affect the smooth function $\tilde{f}_i$, and the agents are risk-neutral with respect to the uncertainty. If we denote by $\xi = \op{col}(\xi^i)_{i\in\mc I} \in \Xi := \prod_{i=1}^N\Xi_i$ the joint tuple of random variables affecting the individual players' costs, our natural extension of the game's pseudogradient under uncertainty becomes 
\begin{equation}\label{eq:tildeF}
	\tilde{F}(\ubold,\xi):= \mathrm{col}\left( \nabla_{u_1}\tilde{f}_{1}(\ubold,\xi^1),\ldots,\nabla_{u_N}\tilde{f}_{N}(\ubold,\xi^N)\right).
\end{equation}

\section{Distributed VR Stochastic Forward-backward-forward schemes}\label{sec:split}

\subsection{Distributed Splitting Scheme}
In this paper, each participant retains local data, including $J_i$, $\setD_i$, $A_{i}$ and $b_{i}$ and is allowed to communicate with trusted neighbors within a certain range. In order to solve the problem of obtaining information about neighbors' decisions, a distributed splitting technique is provided in \cite{YiPav19}.

In addition to its own decision $u_{i}$, each player $i$ keeps a local backup of the dual variables $y_{i}\in\R^{m}_{\geq0}$ in local memory. Based on this, an auxiliary variable $p_{i}\in\R^{m}$ is introduced to push all dual variables towards consensus. To achieve this goal, we define a weighted adjacency matrix $\bW = [w_{ij}]\in\R^{N\times N}$, where the weights $w_{ij}\in[0,1]$. When $w_{ij}>0$, player $i$ may get information from player $j$. For player $i$, $\scrN^{y}_{i}:=\{j\mid w_{ij}>0\}$ denotes the set of the neighbors it can communicate with (also called \emph{operator neighbors}). The Laplacian of $\bW$ is denoted by $\bL=\bD-\bW$, where $\bD:= \diag\left\{D_1,D_2,\dots,D_N\right\} =\diag\left\{\bW\cdot\1_N\right\}$. We impose standard assumptions on the communication technology of the agents. 
\begin{assumption}\label{ass:graph} \em
	The adjacency matrix $\bW$ is irreducible and symmetric.
\end{assumption}
Eigenvalues of $\bL$ satisfy that $0=s_{1}<s_{2}\leq \ldots \leq s_{N}$ given Assumption \ref{ass:graph}. For the maximal degree $\Delta:=\max_{i\in\mc I}D_{i}$, it holds true that $s_{N}\in[\Delta, 2\Delta]$ and $2\Delta\geq \kappa$ \cite{godsil2013}, where $\kappa$ indicates the determinant of $\bL$. We denote $\bar{b}=(b_1,\ldots,b_{N})^{\top}$, $y=\text{col}(y_{1},\ldots,y_{N})$ and $p=\text{col}(p_{1},\ldots,p_{N})$. The state variable $\xbold=(\ubold,p,y)\in\setX=\setD\times\R^{mN}\times\R^{mN}$.
Thus, we can define the following two monotone operators \cite{cui2021relaxed}
\begin{align}
	\label{eq:V}
	\opV(\xbold):&= \left[\begin{array}{c} F(\ubold)+\bA^{\top}y\\
		\bar{\bL}y\\
		\bar{b}+\bar{\bL}(y-p)-\bA\ubold\end{array}\right],\\
	\label{eq:T}
	\opT(\xbold)&:= G(\ubold)\times \{\0_{Nm}\}\times \op{N}_{\R^{mN}_{\geq 0}}(y),
\end{align}
where $\bA:= \diag\{A_1,\ldots,A_N\}$ and $G(\ubold):= \partial g_{1}(u_{1})\times\cdots\times\partial g_{N}(u_{N})$, and $\bar{\bL}:=\bL\otimes \bI_{m}$ is the tensorized Laplacian. The properties of the operators defined in \eqref{eq:V} and \eqref{eq:T} are concluded in the following lemma.
\begin{lemma}\label{lemma_op}
	It holds that:
\begin{itemize}
\item[(a)] The monotone operator $\opV:\setX\to\setX$ is $\ell_{V}=( \ell+2\kappa+\abs{\bA})$-Lipschitz continuous.\\
\item[(b)] $\opT:\setX \rightrightarrows \setX$ is maximally monotone.
\end{itemize}
\end{lemma}
\begin{proof}
	(a) The operator $\opV$ can be divided into $\opV=\opV_{1}+\opV_{2}$, where $\opV_{1}(\xbold)\triangleq \text{col}(F(\ubold),{\bf{0}}_{Nm},\bar{\bL}y+\bar b)$ and $\opV_{2}(\xbold)\triangleq \text{col}(\bA^{\top}y, \bar{\bL}y,-\bA\ubold-\bar{\bL}p)$. From \cite[Proposition 20.23]{BauCom16} and \cite[Corollary 20.28]{BauCom16}, we can see that they are both maximal monotone. Moreover, $\opV_{1}$ is $\ell_{1}=(\ell+\kappa)$-Lipschitz continuous \cite[Lemma 1]{franci2020fbf} and $\opV_{2}$ is $\ell_{2}=(\kappa+|\bA|)$-Lipschitz continuous. Thus, $\opV$ is $\ell_{1}+\ell_{2}=\ell_{V}$-Lipschitz continuous.\\
	(b) See \cite[Lemma 5]{YiPav19}, \cite[Lemma 1]{franci2020fbf}.
\end{proof}

Consequently, we transform the KKT conditions in \eqref{KKT_VI} into the form of a distributed splitting, and the VE of the original game can be obtained by searching for the zeros of $\opV+\opT$ (cf. \cite[Theorem 2]{YiPav19}~or~\cite[Lemma 3]{franci2020fbf}).
\begin{proposition}
	The set 
    $$
    \zer(\opV+\opT):=\{\xbold\in\setX\vert 0\in \opV(\xbold)+\opT(\xbold)\}
    $$
    is equivalent to the VE of the game \eqref{eq:game}. $\hfill \Box$
\end{proposition}

As mentioned in \ref{ssec:stochastic}, for any $\x \in \setX$ we have
\begin{equation}
	\opV(\x) = \EE\left[\tilde{\opV}(\x,\xi)\right],
\end{equation}
where $\tilde{\opV}$ is the random version of the operator in \eqref{eq:V}, i.e., 
\begin{align}
	\label{eq:tildeV}
	\tilde{\opV}(\xbold,\xi):&= \left[\begin{array}{c} \tilde{F}(\ubold,\xi)+\bA^{\top}y\\
		\bar{\bL}y\\
		\bar{b}+\bar{\bL}(y-p)-\bA\ubold\end{array}\right].
\end{align}
\begin{algorithm}[t]
	\caption{Distributed Variance-reduced Stochastic FBF (\bf{DVRSFBF})}\label{alg:DVRSFBF}
	For Player $i$:\\
	(1) Let $t = 0, u_i^0 \in\R^{d_{i}} , y_i^0\in \R_{+}^{m}$ and $p_i^0 \in \R^{m}$. Given $T$, $K$.\\
	(2) While $t<T$, \\
	(3) Receive $u_j^t$ for $j \in \scrN_{i}^{A}$, and set $\bar{F}_{i,t} = \tfrac{\sum_{s=1}^{S_t}\tilde{F}_i(u^t_i,u_j^t, \xi^i_{s,t})}{S_t}$. Let $k=0$, $u_{i,0}^t=u_i^t, y_{i,0}^t=y_i^t, p_{i,0}^t=p_i^t$.\\
	(4) While $k<K$, \\
	(5) Receive $u_{j,k}^t$ for $j \in \scrN_{i}^{A}$, $ y_{j,k}^t$ and $p_{j,k}^t$ for $j \in \scrN_{i}^{y}$ and update 
	$$\begin{aligned}
		u^t_{i,k+\frac{1}{2}} &=\prox_{\gamma_{i}g_{i}}[u^t_{i,k}-\gamma_{i}(\bar{F}_{i,t}+A_{i}^{\top} y^t_{i})]\\
		p^t_{i,k+\frac{1}{2}} &=p^t_{i,k}+\sigma_{i} \sum\nolimits_{j} w_{i,j}(y^t_{j}-y^t_{i})\\
		y^t_{i,k+\frac{1}{2}} &=\Pi_{\R^{m}_{\geq 0}}\{y^t_{i,k}+\tau_{i}(A_{i}u^t_{i}-b_{i})\\
        &+\tau_i\sum\nolimits_{j} w_{i,j}[(p^t_{i}-p^t_{j})-(y^t_{i}-y^t_{j})]\}
	\end{aligned}$$
	
	(6) Receive $u^t_{j,k+\frac{1}{2}}$ for $j \in \scrN_{i}^{A}$, $ y^t_{j,k+\frac{1}{2}}$ and $p^t_{j,k+\frac{1}{2}}$ for $j \in \scrN_{i}^{y}$ and update
	$$\begin{aligned}
		u^t_{i,k+1}=&\, u^t_{i,k+\frac{1}{2}}+\gamma_{i}[\tilde{F}_i(u^t_i,u_j^t,\xi^i_{k+\frac{1}{2},t})\\
        &-\tilde{F}_i(u^t_{i,k+\frac{1}{2}},u^t_{j,k+\frac{1}{2}},\xi^i_{k+\frac{1}{2},t})+A_{i}^{\top} (y^t_{i}-y^t_{i,k+\frac{1}{2}})]\\
		p^t_{i,k+1}=&\,p^t_{i,k+\frac{1}{2}}+\sigma_{i} \sum\nolimits_{j} w_{i,j}[(y^t_{i}-y^t_{j})\\
        &-(y^t_{i,k+\frac{1}{2}}-y^t_{j,k+\frac{1}{2}})]\\
		y^t_{i,k+1}=&\,y^t_{i,k+\frac{1}{2}}-\tau_{i}A_i(u^t_{i}-u^t_{i,k+\frac{1}{2}})\\
        &-\tau_i\sum\nolimits_{j \in \mathcal{N}_{i}^{y}} w_{i,j}[(p^t_{i}-p^t_{j})-(p^t_{i,k+\frac{1}{2}}-p^t_{j,k+\frac{1}{2}})]\\
        &+\tau_i\sum\nolimits_{j \in \mathcal{N}_{i}^{y}} w_{i,j}[(y^t_{i}-y^t_{j})-(y^t_{i,k+\frac{1}{2}}-y^t_{j,k+\frac{1}{2}})]
	\end{aligned}$$
	(7) Set $k\coloneqq k+1$ and go to (4).\\/
	(8) Set $u^{t+1}_i\coloneqq u^t_{i,K}$, $p^{t+1}_i\coloneqq p^t_{i,K}$ and $y^{t+1}_i\coloneqq y^t_{i,K}$. Let $t\coloneqq t+1$ and go to (2).
\end{algorithm}

\subsection{Distributed algorithm}
The main contribution of this paper lies in the development of a stochastic variance-reduced distributed algorithm for computing a VE of the SGNE. Contrary to previous studies, our algorithmic design takes inspiration from the machine learning literature, and includes an implicit double-loop scheme with different ways to sample the stochastic operator. In particular, our design translates the SVRG technology of \cite{johnson2013accelerating}, originally designed for finite-sum convex optimization problems, to stochastic operator splitting problems. The main ideas behind the algorithmic design is as follows: Since the operators derived from GNEPs do usually not display the technical cocoercivity condition, an extragradient type algorithm is sought for (see \cite{grammatico2018} for an example on what can go wrong in game dynamics without cocoercivity). The type of extragradient method employed is the forward-backward-forward splitting template of \cite{Tse00}. The next basic design element is the variance reduction technology. Our algorithm features two iteration counters $t$ and $k$; The index $t$ measures the iterations of an ``outer loop'', at which a costly mini-batch Monte-Carlo estimator of the stochastic pseudo-gradient is computed. The batch size of this estimator $S_{t}$ is dynamically adjusted to control the variance. Such variance reduction is especially applicable to simulation-based scenarios, where repeated samples are easy to generate \cite{ByrChiNocWu12,Bot2021,cui2023variance,lei2022distributed}. The iteration counter $k$ is the index of an ``inner loop''. Within this inner loop only single-sample estimators of the game's pseudogradient are determined. This phase of the scheme therefore has low computational costs, assuming that function evaluation is within reasonable computational budget. All these loops contain fully parallel computations in which each agent has its own algorithmic parameters. For the convenience of analysis, we collect all these parameters in the block matrix
\begin{equation}\label{eq:Phi}
	\Phi= \diag(\gamma^{-1},\sigma^{-1},\tau^{-1}).
\end{equation}
The state variable of the algorithm is denoted as $\x^t = (\ubold^t,p^t,y^t)$, where $\x^t$ is used to calculate the stochastic variance-reduced operator
\begin{equation}\label{eq:barV}
	\bar{\opV}(\x^t)=\frac{\Sigma_{j=1}^{S_t}\tilde{\opV}(\x^t,\xi_{j,t})}{S_t}.
\end{equation}

Let $\z^t_k$ be the state variable of the inner steps at iteration $t$ and we have $\z^t_0 = \x^t$. The inner loop follows the form of modified FBF splitting, which can be succinctly expressed as:
\begin{equation}\label{eq:alg}
	\left\{\begin{array}{l}
		\z^t_{k+\frac{1}{2}}= \mathrm{J}_{\Phi^{-1}\opT}(\z^t_k-\Phi^{-1}{\bar{\opV}}(\x^t)), \\
		\z^t_{k+1}= \z^t_{k+\frac{1}{2}}-\Phi^{-1}({\tilde{\opV}}(\z^t_{k+\frac{1}{2}},\xi_{k+\frac{1}{2},t})\\
		\qquad\quad-{\tilde{\opV}}(\x^t,\xi_{k+\frac{1}{2},t})),
	\end{array}\right.
\end{equation}
where $\mathrm{J}_{\Phi^{-1}\opT}= (\Id+ \Phi^{-1}\opT)^{-1}$ represents the resolvent of the operator $\Phi^{-1}\opT$ and $\op{Id}$ denotes the identity operator. The development of DVRSFBF and the convergence of variance reduction schemes are discussed in the next section.

\section{Convergence analysis}	\label{sec:convergence}
\subsection{Background and assumptions}\label{sec:4.1}
We formally define such a variance-reduced splitting scheme next. We begin by introducing suitable monotonicity and Lipschitzian requirements on $\opV$ and $\tilde{\opV}(\bullet,\xi)$, respectively. 
	
	\begin{assumption}\label{ass:lipmon} \em 		
		(a) The mapping ${\tilde{\opV}}(\bullet,\xi)$ is $L(\xi)$-Lipschitz continuous, where $L(\xi)\in L^{2}(\Omega,\PP)$
		
		(b) The mapping $\opV$ is monotone and single-valued on $\R^n$;
		
		(c) The mapping $\opT$ is maximal monotone on $\R^n$.
	\end{assumption}
	\begin{remark}
		As Assumptions \ref{ass:convex}, \ref{ass:X}, \ref{ass:GM} and \ref{ass:graph} hold, Assumption \ref{ass:lipmon}(b) and (c) can be deduced from Lemma \ref{lemma_op}.
	\end{remark}
	
	We now formally define the information structure. 
	\begin{definition}
		\noindent {For any $t \, \geq \, 0$}, suppose $\mathcal{F}_t$ denotes the history up to iteration $t$, where $\mathcal{F}_0  \triangleq \sigma(\x^0)$ and 
		\begin{align*} 
			\mathcal{F}_t & \triangleq
			\sigma\left(\x^0,\{\xi_{j,0}\}_{j=1}^{S_0},\{\xi_{k+\frac{1}{2},0}\}_{k=0}^{K-1}, \cdots, \{\xi_{j,t-1}\}_{j=1}^{S_{t-1}},\right.\\ &\left.\{\xi_{k+\frac{1}{2},t-1}\}_{k=0}^{K-1}\right).
		\end{align*}
		We denote the history at outer iteration $t \geq 0$ and inner iteration $K \geq k \ge 0$  as $\mathcal{F}_{k,t}$ and is defined as  follows.
		\begin{align*} 		
        &\mathcal{F}_{0,t} \, \triangleq \, \sigma(\mathcal{F}_t \, \cup \,  \sigma(\left\{\xi_{j,t}\right\}_{j=1}^{S_{t}}))\mbox{ and }\\
			&\mathcal{F}_{k,t} \,  \triangleq\,  \sigma\left(\mathcal{F}_t \, \cup \,
			\sigma\left(\left\{ \{\xi_{j,t}\}_{j=1}^{S_{t}}, \xi_{\frac{1}{2},t},\cdots,\xi_{k-\frac{1}{2},t} \right\}\right)\right).
		\end{align*}
	\end{definition}
	
	\medskip

	In our proposed variant of the SVRG framework, an average gradient $\bar{\opV}$ is calculated to approximate the expectation-valued $\opV$. We impose the following bias and moment assumptions on the stochastic error
    \begin{equation}\label{eq:def_eps}
        \eps_{k,t}(x)\triangleq\tilde{\opV}(x,\xi_{k,t})-\opV(x).
    \end{equation}
	\begin{assumption}\label{moment_biased} \em Let $\left\{\x^t\right\}$ be the stochastic process generated by ({\bf DVRSFBF}). Suppose $\{\xi_{j,t}\}_{j=1}^{S_{t}}$ are i.i.d samples for $t\geq 0$. For all $t, k$ and $j \geq 0$, there exist scalars $\nu, b > 0$ such that the following conditions apply: 
		
		(a)  The conditional means $\| \mathbb{E}[\eps_{j,t}(\x^t)  \mid  {\cal F}_t ] \|  \le \frac{b}{\sqrt{S_t}}$, $\| \mathbb{E}[\eps_{k+\frac{1}{2},t}(\z^t_{k+\frac{1}{2}}) \mid {\cal F}_{k,t} ] \|  \le \frac{b}{\sqrt{S_t}},\, \| \mathbb{E}[\eps_{k+\frac{1}{2},t}(\x^t) \mid {\cal F}_{k,t} ] \|  \le \frac{b}{\sqrt{S_t}}$ a.s.;
		
		(b) The conditional second moments $\mathbb{E}[\|\eps_{j,t}(\x^t)\|^2\mid  {\cal F}_t] \le \nu^2,\, \mathbb{E}[\|\eps_{k+\frac{1}{2},t}(\z^t_{k+\frac{1}{2}})\|^2 \mid {\cal F}_{k,t} ] \le \nu^2,\,  \mathbb{E}[\|\eps_{k+\frac{1}{2},t}(\x^t)\|^2 \mid {\cal F}_{k,t} ] \le \nu^2$, a.s .
	\end{assumption}
	
	Note that the first assumption is a weakening of the standard conditional unbiasedness assumption popular in stochastic optimization. Our results hold under a weaker notion of asymptotically unbiased oracles. This requirement is rather mild and is imposed in many simulation-based optimization schemes.
	We conclude this subsection with the Robbins-Siegmund Lemma crucial for proving a claim of almost sure convergence in strongly monotone regimes.
	
	\begin{lemma}\cite[Lemma 10]{Pol87}
    \label{as_recur}\em Let $\{v_t\}$ be a sequence of nonnegative random variables adapted to the $\sigma$-algebra ${\cal F}_t$ such that for $t \geq 0$,  
		\begin{align}
			\mathbb{E}\left[\, v_{t+1} \, \mid \, {\cal F}_t \, \right] \, \leq \, (1-u_t) v_t + \beta_t, \mbox{ a.s. } 
		\end{align}
		where $0 \leq u_t \leq 1$, $\beta_{t}\geq 0$, $\sum_{t=0}^{\infty} u_t = \infty$, $\sum_{t=0}^{\infty} \beta_t < \infty$, and $\lim_{t \to \infty} \tfrac{\beta_t}{u_t} = 0$. Then $v_t \xrightarrow{t \to \infty} 0$ almost surely.
	\end{lemma}
	
	\subsection{Convergence analysis under strong monotonicity} \label{sec:4.2}
	In this subsection, we derive a.s. convergence guarantees and rate statements.
	\begin{lemma}\label{egt} \em
		Consider a sequence $\{\x^t\}$ generated by ({\bf DVRSFBF}). Let Assumptions~\ref{ass:convex}-\ref{moment_biased} hold. Suppose that the operator $\opV$ is $\mu$-strongly monotone. {Then there exist $q < 1, \rho < 1-q$, and $0<\mathsf{c}<\frac{\mu}{3}$ such that for any $t \ge 0$, 
		\begin{align} \label{rec:cond_exp_smon}
			\notag\mathbb{E}[\|\x^{t+1}-\x^*\|^2_\Phi&\mid \mathcal{F}_t] \le \left(q^{K}+\frac{1-q^{K}}{1-q}\rho\right)\|\x^{t}-\x^*\|^2_\Phi\\
			&+\frac{1-q^{K}}{1-q} \left( \frac{3b^2+\nu^2}{\mathsf{c}S_t} \right).
		\end{align}}
	\end{lemma}
	\begin{proof}
		From the definition of $\z^t_{k+\frac{1}{2}}$, we have
		\begin{align*} 
			\z^t_{k+\frac{1}{2}}+\Phi^{-1} \vv^t_{k+\frac{1}{2}} & =\z^t_k-\Phi^{-1}(\opV(\x^t)+\bar{\epsilon}_t),
		\end{align*}
		where $\bar{\epsilon}_t \, = \, \bar{\opV}(\x^t) - \opV(\x^t)$ and $\vv^t_{k+\frac{1}{2}} \in \opT(\z^t_{k+\frac{1}{2}})$. Since $0 \in \opV(\x^*)+\opT(\x^*)$, there exists $\vv^* \in \opT(\x^*)$ such that
		\begin{align} \label{df2}
			\opV(\x^*)+\vv^*=0.
		\end{align}
		From $\z^t_{k+1} = \z^t_{k+\frac{1}{2}}-\Phi^{-1}({\tilde{\opV}}(\z^t_{k+\frac{1}{2}},\xi_{k+\frac{1}{2},t})-{\tilde{\opV}}(\x^t,\xi_{k+\frac{1}{2},t}))$, we obtain 
        \begin{align*}
            \|\z^t_k&-\x^*\|^2_\Phi=\|\z^t_k-\z^t_{k+\frac{1}{2}}\|^2_\Phi+\|\z^t_{k+\frac{1}{2}}-\z^t_{k+1}\|^2_\Phi\\
                &~+\|\z^t_{k+1}-\x^*\|^2_\Phi-2\|\z^t_{k+\frac{1}{2}}-\z^t_{k+1}\|^2_\Phi\\
			&~+2\inner{\z^t_k-\z^t_{k+\frac{1}{2}},\z^t_{k+\frac{1}{2}}-\x^*}_\Phi\\
                &~+2\inner{\z^t_{k+\frac{1}{2}}-\z^t_{k+1},\z^t_{k+\frac{1}{2}}-\x^*}_\Phi\\
                &=\|\z^t_k-\z^t_{k+\frac{1}{2}}\|^2_\Phi-\|\z^t_{k+\frac{1}{2}}-\z^t_{k+1}\|^2_\Phi\\
                &~+2\inner{\z^t_k-\z^t_{k+1},\z^t_{k+\frac{1}{2}}-\x^*}_\Phi+\|\z^t_{k+1}-\x^*\|^2_\Phi.
		\end{align*}
        Hence, 
        \begin{align}\label{eq18}
        \|&\z^t_k-\x^*\|^2_\Phi=\|\z^t_k-\z^t_{k+\frac{1}{2}}\|^2_\Phi\notag\\
        &~ -\|{\tilde{\opV}}(\z^t_{k+\frac{1}{2}},\xi_{k+\frac{1}{2},t})-{\tilde{\opV}}(\x^t,\xi_{k+\frac{1}{2},t})\|^2_{\Phi^{-1}}\notag\\
	&~+\|\z^t_{k+1}-\x^*\|^2_\Phi +2\inner{\vv^{t}_{k+1/2}, \z^t_{k+\frac{1}{2}}-\x^*}\notag\\
        &~+ 2\inner{{\tilde{\opV}}(\z^t_{k+\frac{1}{2}},\xi_{k+\frac{1}{2},t})-{\tilde{\opV}}(\x^t,\xi_{k+\frac{1}{2},t}), \z^t_{k+\frac{1}{2}}-\x^*}\notag\\
        &~+2\inner{\opV(\x^t)+\bar{\epsilon}_t, \z^t_{k+\frac{1}{2}}-\x^*}.
        \end{align}
        Since $\tilde{\opV}(\bullet,\xi)$ is $L(\xi)$-Lipschitz continuous, we obtain the following in regards to the $\Phi^{-1}$-operator norm:
        \begin{align*}
          \|{\tilde{\opV}}(\z^t_{k+\frac{1}{2}},&\xi_{k+\frac{1}{2},t})-{\tilde{\opV}}(\x^t,\xi_{k+\frac{1}{2},t})\|^2_{\Phi^{-1}}\\ 
          &\leq (L({\xi_{k+\frac{1}{2},t}))}^2\lambda_{\text{max}}(\Phi^{-1})^2 \|\z^t_{k+\frac{1}{2}}-\x^t\|^2_{\Phi}\\
          &=(L_{\|\cdot\|_{\Phi^{-1}}}(\xi_{k+\frac{1}{2},t}))^2\|\z^t_{k+\frac{1}{2}}-\x^t\|^2_{\Phi},
        \end{align*}
        where $\lambda_{\text{max}}(\Phi^{-1})$ denotes the largest eigenvalue of $\Phi^{-1}$ and $L_{\|\cdot\|_{\Phi^{-1}}}(\xi)\triangleq L(\xi)\lambda_{\text{max}}(\Phi^{-1})$.
		Following \eqref{eq18}, we deduce 
		\begin{align}
			\notag \|&\z^t_{k+1}-\x^*\|^2_\Phi \le\|\z^t_k-\x^*\|^2_\Phi-\|\z^t_k-\z^t_{k+\frac{1}{2}}\|^2_\Phi\\
                \notag&+{(L_{\|\cdot\|_{\Phi^{-1}}}({\xi_{k+\frac{1}{2},t}))}^2}\|\z^t_{k+\frac{1}{2}}-\x^t\|^2_{\Phi}\\
			\notag& -2\inner{\vv^t_{k+\frac{1}{2}}+\opV(\z^t_{k+\frac{1}{2}}),\z^t_{k+\frac{1}{2}}-\x^*}\\
			\notag&-2\inner{{\tilde{\opV}}(\z^t_{k+\frac{1}{2}},\xi_{k+\frac{1}{2},t})-{\tilde{\opV}}(\x^t,\xi_{k+\frac{1}{2},t})-\opV(\z^t_{k+\frac{1}{2}})+\bar{\epsilon}_t\\
			&+\opV(\x^t),\z^t_{k+\frac{1}{2}}-\x^*}. \label{eq29}
		\end{align}
		Since $0 \in \opV(\x^*)+\opT(\x^*)$, the following bound holds by strong monotonicity of $\opV$ and \eqref{df2}:
		\begin{align}
			\inner{\vv^t_{k+\frac{1}{2}}+\opV(\z^t_{k+\frac{1}{2}}),\z^t_{k+\frac{1}{2}}-\x^*} \ge \mu\|\z^t_{k+\frac{1}{2}}-\x^*\|^2. \label{eq19}
		\end{align}
		Substituting \eqref{eq19} into \eqref{eq29}, we obtain
		\begin{align*}
			\|&\z^t_{k+1}-\x^*\|^2_\Phi \le\|\z^t_k-\x^*\|^2_\Phi-\|\z^t_k-\z^t_{k+\frac{1}{2}}\|^2_\Phi\\
                &+{(L_{\|\cdot\|_{\Phi^{-1}}}({\xi_{k+\frac{1}{2},t}))^2}}\|\z^t_{k+\frac{1}{2}}-\x^t\|^2_{\Phi}-2\mu\|\z^t_{k+\frac{1}{2}}-\x^*\|^2\\
			&-2\inner{{\tilde{\opV}}(\z^t_{k+\frac{1}{2}},\xi_{k+\frac{1}{2},t})-{\tilde{\opV}}(\x^t,\xi_{k+\frac{1}{2},t})+\opV(\x^t)\\
			&-\opV(\z^t_{k+\frac{1}{2}})+\bar{\epsilon}_t,\z^t_{k+\frac{1}{2}}-\x^*}.
		\end{align*}
		The next inequalities follow from $\|u+v+w\|^2 \le 3\|u\|^2+3\|v\|^2+3\|w\|^2$ {and $-2\|u\|^2 \leq -\|u+v\|^2 + 2\|v\|^2.$}
		\begin{align*}
			\|&\z^t_{k+1}-\x^*\|^2_\Phi  \le \|\z^t_k-\x^*\|^2_\Phi-\|\z^t_k-\z^t_{k+\frac{1}{2}}\|^2_\Phi\\
                &~+3(L_{\|\cdot\|_{\Phi^{-1}}}(\xi_{k+\frac{1}{2},t}))^2\|\z^t_{k}-\x^*\|^2_{\Phi}-2\mu\|\z^t_{k+\frac{1}{2}}-\x^*\|^2\\
                &~+3(L_{\|\cdot\|_{\Phi^{-1}}}(\xi_{k+\frac{1}{2},t}))^2\|\x^t-\x^*\|^2_{\Phi}\\
                &~+3(L_{\|\cdot\|_{\Phi^{-1}}}(\xi_{k+\frac{1}{2},t}))^2\|\z^t_k-\z^t_{k+\frac{1}{2}}\|^2_{\Phi}\\
			&~-2\inner{{\tilde{\opV}}(\z^t_{k+\frac{1}{2}},\xi_{k+\frac{1}{2},t})-{\tilde{\opV}}(\x^t,\xi_{k+\frac{1}{2},t})+\opV(\x^t)\\
                &~-\opV(\z^t_{k+\frac{1}{2}})+\bar{\epsilon}_t,\z^t_{k+\frac{1}{2}}-\x^*}\\           
                &\le \|\z^t_k-\x^*\|^2_\Phi+3(L_{\|\cdot\|_{\Phi^{-1}}}(\xi_{k+\frac{1}{2},t}))^2\|\z^t_{k}-\x^*\|^2_{\Phi}\\
			&~-\|\z^t_k-\z^t_{k+\frac{1}{2}}\|^2_\Phi+3(L_{\|\cdot\|_{\Phi^{-1}}}(\xi_{k+\frac{1}{2},t}))^2\|\x^t-\x^*\|^2_{\Phi}\\
			&~ +3(L_{\|\cdot\|_{\Phi^{-1}}}(\xi_{k+\frac{1}{2},t}))^2\|\z^t_k-\z^t_{k+\frac{1}{2}}\|^2_{\Phi}\\
			&~-2\inner{{\tilde{\opV}}(\z^t_{k+\frac{1}{2}},\xi_{k+\frac{1}{2},t})-{\tilde{\opV}}(\x^t,\xi_{k+\frac{1}{2},t})+\opV(\x^t)-\opV(\z^t_{k+\frac{1}{2}})\\
                &~+\bar{\epsilon}_t,\z^t_{k+\frac{1}{2}}-\x^*}-\mu\|\z^t_k-\x^*\|^2+2\mu\|\z^t_{k+\frac{1}{2}}-\z^t_k\|^2.
		\end{align*}
		We define $\alpha\ \triangleq \lambda_{\text{max}}(\Phi^{-1})$ and the above inequality becomes
		\begin{align*}
			\|&\z^t_{k+1}-\x^*\|^2_\Phi \le (1-\alpha\mu+3(L_{\|\cdot\|_{\Phi^{-1}}}(\xi_{k+\frac{1}{2},t}))^2)\|\z^t_k-\x^*\|^2_\Phi\\
			&-(1-2\alpha\mu-3(L_{\|\cdot\|_{\Phi^{-1}}}(\xi_{k+\frac{1}{2},t}))^2)\|\z^t_k-\z^t_{k+\frac{1}{2}}\|^2_\Phi\\
			&+3(L_{\|\cdot\|_{\Phi^{-1}}}(\xi_{k+\frac{1}{2},t}))^2\|\x^t-\x^*\|^2_{\Phi} -2\inner{\opV(\z^t_{k+\frac{1}{2}},\xi_{k+\frac{1}{2},t})\\
			&-\opV(\x^t,\xi_{k+\frac{1}{2},t})+\opV(\x^t) -\opV(\z^t_{k+\frac{1}{2}})+\bar{\epsilon}_t,\z^t_{k+\frac{1}{2}}-\x^*}.
		\end{align*}
		Taking expectations on both sides {conditioned on $\mathcal{F}_t$}, invoking the tower law of conditional expectation and using \ref{ass:lipmon}(a) we obtain the following inequality.
		\begin{align*}
			&\mathbb{E}\left[\, \|\z^t_{k+1}-\x^*\|^2_\Phi \, \mid \, \mathcal{F}_t\, \right]\\
                &\le \mathbb{E}\left[ (1-\alpha\mu+3{L_{\|\cdot\|_{\Phi^{-1}}}(\xi_{k+\frac{1}{2},t})}^2)\|\z^t_k-\x^*\|^2_\Phi\mid \, \mathcal{F}_t \, \right]\\
			    &~-\mathbb{E}\left[ (1-2\alpha\mu-3L_{\|\cdot\|_{\Phi^{-1}}}(\xi_{k+\frac{1}{2},t})^2)\|\z^t_k-\z^t_{k+\frac{1}{2}}\|^2_\Phi\mid \, \mathcal{F}_t \, \right]\\
                &~-2\mathbb{E}\left[\, \inner{{\tilde{\opV}}(\z^t_{k+\frac{1}{2}},\xi_{k+\frac{1}{2},t})-\opV(\z^t_{k+\frac{1}{2}}),\z^t_{k+\frac{1}{2}} -\x^{*}}\, \mid \, \mathcal{F}_t\, \right]\\
			&~ -2\mathbb{E}\left[\inner{{\tilde{\opV}}(\x^t,\xi_{k+\frac{1}{2},t})-\opV(\x^t),\z^t_{k+\frac{1}{2}}-\x^*} \mid \, \mathcal{F}_t\, \right]\\
			&~-2\mathbb{E}\left[\, \inner{\bar{\epsilon}_t,\z^t_{k+\frac{1}{2}}-\x^*}\, \mid \, \mathcal{F}_t\, \right]\\ 
                &~+3L_{\|\cdot\|_{\Phi^{-1}}}(\xi_{k+\frac{1}{2},t})^2\|\x^t-\x^*\|^2_{\Phi}\\
			&\le\mathbb{E}\left[ (1-\alpha\mu+3{L_{\|\cdot\|_{\Phi^{-1}}}(\xi_{k+\frac{1}{2},t})}^2)\|\z^t_k-\x^*\|^2_\Phi\, \mid \, \mathcal{F}_t \, \right]\\
                &~-\mathbb{E}\left[ (1-2\alpha\mu-3L_{\|\cdot\|_{\Phi^{-1}}}(\xi_{k+\frac{1}{2},t})^2)\|\z^t_k-\z^t_{k+\frac{1}{2}}\|^2_\Phi\mid \, \mathcal{F}_t \, \right]\\
			&~+2\mathbb{E}\left[\left\|\mathbb{E}\left[ \tilde{\opV}(\z^t_{k+\frac{1}{2}},\xi_{k+\frac{1}{2},t})-\opV(\z^t_{k+\frac{1}{2}}) \mid  \mathcal{F}_{k,t}\right]\right\|\mid  \mathcal{F}_t\right]\\
                &~\cdot \mathbb{E}\left[\|\z^t_{k+\frac{1}{2}}-\x^*\| \mid  \mathcal{F}_t\right]+2\mathbb{E}\left[\|\bar{\epsilon}_t\|\|\z^t_{k+\frac{1}{2}}-\x^*\|\, \mid \, \mathcal{F}_t\, \right]\\
			&~+2\mathbb{E}\left[\left\|\mathbb{E}\left[{\tilde{\opV}}(\x^t,\xi_{k+\frac{1}{2},t})-\opV(\x^t) \mid  \mathcal{F}_{k,t} \right]\right\|\mid  \mathcal{F}_t \right]\\
                &~\cdot \mathbb{E}\left[\|\z^t_{k+\frac{1}{2}}-\x^*\|\mid  \mathcal{F}_t \right]+3L_{\|\cdot\|_{\Phi^{-1}}}(\xi_{k+\frac{1}{2},t})^2\|\x^t-\x^*\|^2_{\Phi}.\\                
		\end{align*}
        {\begin{align*}
             &\le\mathbb{E}\left[ (1-\alpha\mu+3{L_{\|\cdot\|_{\Phi^{-1}}}(\xi_{k+\frac{1}{2},t})}^2)\|\z^t_k-\x^*\|^2_\Phi\, \mid \, \mathcal{F}_t \, \right]\\
                &~-\mathbb{E}\left[ (1-2\alpha\mu-3L_{\|\cdot\|_{\Phi^{-1}}}(\xi_{k+\frac{1}{2},t})^2)\|\z^t_k-\z^t_{k+\frac{1}{2}}\|^2_\Phi\mid \, \mathcal{F}_t \, \right]\\
			&~+\frac{1}{\mathsf{c}}\mathbb{E}\left[\left\|\mathbb{E}\left[ \tilde{\opV}(\z^t_{k+\frac{1}{2}},\xi_{k+\frac{1}{2},t})-\opV(\z^t_{k+\frac{1}{2}}) \mid  \mathcal{F}_{k,t}\right]\right\|^2\mid  \mathcal{F}_t\right]\\
                &~+ \mathsf{c} \left(\mathbb{E}\left[\|\z^t_{k}-\x^*\|^2 \mid  \mathcal{F}_t\right] + \mathbb{E}\left[\|\z^t_{k+\frac{1}{2}}-\z^t_{k}\|^2 \mid  \mathcal{F}_t\right] \right)\\
			&~+\frac{1}{\mathsf{c}}\mathbb{E}\left[\left\|\mathbb{E}\left[{\tilde{\opV}}(\x^t,\xi_{k+\frac{1}{2},t})-\opV(\x^t) \mid  \mathcal{F}_{k,t} \right]\right\|^2\mid  \mathcal{F}_t \right]\\
                &~+ \mathsf{c}\left(\mathbb{E}\left[\|\z^t_{k}-\x^*\|^2\mid  \mathcal{F}_t \right] + \mathbb{E}\left[\|\z^t_{k+\frac{1}{2}}-\z^t_{k}\|^2\mid  \mathcal{F}_t \right]\right)\\
                &~+ \mathsf{c}\left(\mathbb{E}\left[\|\z^t_{k}-\x^*\|^2\mid  \mathcal{F}_t \right] + \mathbb{E}\left[\|\z^t_{k+\frac{1}{2}}-\z^t_{k}\|^2\mid  \mathcal{F}_t \right]\right)\\
                &~+ \frac{1}{\mathsf{c}}\mathbb{E}\left[\|\bar{\epsilon}_t\|^2\, \mid \, \mathcal{F}_t\, \right] +3L_{\|\cdot\|_{\Phi^{-1}}}(\xi_{k+\frac{1}{2},t})^2\|\x^t-\x^*\|^2_{\Phi},\\
        \end{align*}
        where $\mathsf{c}$ is a positive constant.}
        
		Since $L(\xi)\in L^{2}(\Omega,\PP)$ we set $\mathbb{E}[L(\xi_{k+\frac{1}{2},t})]\triangleq \hat{L}$. {By choosing such a $\mathsf{c}$ that $0<\mathsf{c}<\frac{\mu}{3}$, and an $\alpha$ that $0<\alpha<\min\{\tfrac{\mu-3\mathsf{c}}{6\hat{L}^2},\tfrac{\sqrt{(2\mu+3\mathsf{c})^2+12\hat{L}^2}-(2\mu+3\mathsf{c})}{6\hat{L}^2}\}$} we ensure that {$q \, \triangleq \, (1-\alpha(\mu-3\mathsf{c})+3\alpha^2\hat{L}^2) < 1, \rho \, \triangleq \, 3\alpha^2 \hat{L}^2 < 1-q$, and $(1-\alpha(2\mu+3\mathsf{c})-3\alpha^2\hat{L}^2) > 0$.} Invoking the tower property of the expectation we obtain the following recursion.  
        {
        \begin{align*}
			&\mathbb{E}[\|\z^t_{k+1}-\x^*\|^2_\Phi \, \mid \, \mathcal{F}_t \, ]\\ &\le\underbrace{(1-\alpha(\mu-3\mathsf{c})+3\alpha^2\hat{L}^2) }_q\mathbb{E}\left[\|\z^t_k-\x^*\|^2_\Phi \right.
			\left.\mid \,  \mathcal{F}_t\, \right]\\ &+\underbrace{3\alpha^2\hat{L}^2}_\rho\|\x^t-\x^*\|^2_\Phi+ \left( \frac{3b^2+\nu^2}{\mathsf{c}S_t} \right) \notag \\
			&\le q^{k+1}\|\z^t_0-\x^*\|^2_\Phi+\frac{1-q^{k+1}}{1-q}\rho\|\x^t-\x^*\|^2_\Phi\\
			&+\frac{1-q^{k+1}}{1-q} \left( \frac{3b^2+\nu^2}{\mathsf{c}S_t} \right). 
		\end{align*}
        }
		Consequently, it follows that
        {
        \begin{align}
			\notag\mathbb{E}[\|\x^{t+1}-\x^*\|^2_\Phi&\mid \mathcal{F}_t] \le \left(q^{K}+\frac{1-q^{K}}{1-q}\rho\right)\|\x^{t}-\x^*\|^2_\Phi\\
			&+\frac{1-q^{K}}{1-q} \left( \frac{3b^2+\nu^2}{\mathsf{c}S_t} \right). \label{tr1}
		\end{align}	
        }
	\end{proof} 

	\begin{theorem}[{\bf a.s. convergence of ({\bf DVRSFBF})}] \label{as_DVRSFBF} \em
		Let $\{\x^t\}_{t}$ be the stochastic process generated by ({\bf DVRSFBF}). Let Assumptions~\ref{ass:convex}-\ref{moment_biased} hold. Suppose operator $\opV$ is $\mu$-strongly monotone. Suppose $\sum_{t=1}^{\infty}\tfrac{1}{S_t} < \infty$. Let $\alpha = \lambda_{\max}(\Phi^{-1})$. Suppose $\mathsf{c}$ and $\alpha$ are sufficiently small such that $0<\mathsf{c}<\frac{\mu}{3}$ and $0<\alpha<\min\{\tfrac{\mu-3\mathsf{c}}{6\hat{L}^2},\tfrac{\sqrt{(2\mu+3\mathsf{c})^2+12\hat{L}^2}-(2\mu+3\mathsf{c})}{6\hat{L}^2}\}$. Then $\x^t \xrightarrow [a.s.]{t \to \infty} \x^*$. 	
	\end{theorem}
	\begin{proof}
		{As $0<\mathsf{c}<\frac{\mu}{3}$ and $0<\alpha<\min\{\tfrac{\mu-3\mathsf{c}}{6\hat{L}^2},\tfrac{\sqrt{(2\mu+3\mathsf{c})^2+12\hat{L}^2}-(2\mu+3\mathsf{c})}{6\hat{L}^2}\}$, we have $q \, \triangleq \, (1-\alpha(\mu-3\mathsf{c})+3\alpha^2\hat{L}^2) < 1$ and $\rho \, \triangleq \, 3\alpha^2 \hat{L}^2 < 1-q$.} Therefore we obtain $q^{K}+\frac{1-q^{K}}{1-q}\rho<1$. Considering $\sum_{t=1}^{\infty}\tfrac{1}{S_t} < \infty$, we introduce Lemma \ref{as_recur}. Thus, we have $\norm{\x^{t}-\x^*}^2 \rightarrow 0$ in an a.s. sense as $ t \rightarrow \infty$.
	\end{proof}
	
	Next we examine the linear convergence of ({\bf DVRSFBF}) under strong monotonicity.  
	
	\begin{proposition}[{\bf Linear convergence of ({\bf DVRSFBF})}]\label{ratePvs1} 
		\em Let Assumptions~\ref{ass:convex}-\ref{moment_biased} hold. Suppose operator $\opV$ is $\mu$-strongly monotone. Let $\{\x^t\}_{t}$ denote the stochastic process generated by ({\bf DVRSFBF}), and $\{\x^*\}=\zer(\opV+\opT)$. Define $D^* \triangleq \|\x^0-\x^*\|_\Phi$. Then the following hold.
        \begin{itemize}
            \item[(a)] {$S_{t}=\lfloor {\eta^{-(t+1)}} \rfloor$} where $0<\eta<1$ and $\delta\triangleq q^{K}+\frac{1-q^{K}}{1-q}\rho < 1$. Then $\mathbb{E}[\|\x^t-\x^*\|^2_\Phi] \leq \tilde{G} \tilde{\eta}^t$ where $\tilde G > 0$  and $\tilde \eta = \max\{\delta,\eta\}$ if $\delta \neq \eta$ and $\tilde \eta \in (\delta,1)$ if $\delta = \eta$. 
		\item [(b)] Suppose $\x^{T+1}$ is such that $\mathbb{E}[\|\x^{T+1}-\x^*\|^2_\Phi] \le \epsilon$. {Then $\sum_{t=1}^TS_t \le
		\mathcal{O}\left(\frac{1}{\epsilon}\right).$}
        \end{itemize}
	\end{proposition}
	\begin{proof}
		(a) Suppose $\hat{G}\triangleq \frac{1-q^{K}}{1-q} \left( \frac{3b^2+\nu^2}{\mathsf{c}}\right) $. According to \eqref{tr1},  we obtain the following:
		{\begin{align}
			\mathbb{E}[\|\x^{t+1}&-\x^*\|^2_\Phi]\le \delta\mathbb{E}[\|\x^t-\x^*\|^2_\Phi]+\tfrac{\hat{G}}{S_{t}}. \label{eq58p-t}
		\end{align}}
		Recall that $S_t$ can be bounded as seen next.
		\begin{align}
			S_t= \lfloor \eta^{-(t+1)} \rfloor \ge \left\lceil \tfrac{1}{2}\eta^{-(t+1)} \right\rceil \ge \tfrac{1}{2}\eta^{-(t+1)}. \label{eq59p-t}
		\end{align}
		We now consider three cases.
		
		\noindent  (i): $\delta<\eta<1$. Using \eqref{eq59p-t} in \eqref{eq58p-t} and defining $\bar G=2\hat{G}$, $\tilde G \triangleq (D^*+\tfrac{\bar G}{1-\delta/\eta})$, we obtain
		\begin{align}
			\notag\mathbb{E}[\|&\x^{t+1}-\x^*\|^2_\Phi] \le \delta\mathbb{E}[\|\x^t-\x^*\|^2_\Phi]+\tfrac{\hat{G}}{S_{t}}\\
			&\notag\leq\delta\mathbb{E}[\|\x^t-\x^*\|^2_\Phi]+\bar G\eta^{t+1} \\
			&\notag\le \delta^{t+1}{\|\x^0-\x^*\|_\Phi}+\bar G\sum_{j=1}^{t+1}\delta^{t+1-j}\eta^j 
			\\&\le D^*\delta^{t+1}+\bar G\eta^{t+1}\sum_{j=1}^{t+1}(\tfrac{\delta}{\eta})^{t+1-j}\le \tilde{G}{\eta}^{t+1}.\label{ineq-linear}
		\end{align}
		
		\noindent  (ii): $\eta<\delta<1$. Akin to (i) and defining $\tilde G$ appropriately,  $\mathbb{E}[\|\x^{k+1}-\x^*\|^2_\Phi] \le\tilde{G}{\delta}^{t+1}$. \\
		\noindent (iii): $\eta=\delta<1$. Proceeding similarly we obtain
		\begin{align}
			\notag\mathbb{E}[\|&\x^{t+1}-\x^*\|^2_\Phi] \le \delta^{t+1}\mathbb{E}[\|\x^0-\x^*\|^2_\Phi]+\bar{G}\sum_{j=1}^{t+1}\delta^{t+1}\\&\notag\le D^*\delta^{t+1}+\bar{G}\sum_{j=1}^{t+1}\delta^{t+1}\\ 
			\notag & = D^*\delta^{t+1}+\bar{G}(t+1)\delta^{t+1} 
			\overset{\tiny \mbox{\cite[Lemma~4]{ahmadi2016analysis}}}
			{\le} D^*\delta^{t+1} + G' \tilde{\eta}^{t+1},\label{eq65p-t}
		\end{align}
		\noindent where $\tilde{\eta} \in (\delta,1)$ and $G' > \tfrac{\bar{G}}{\ln(\tilde{\eta}/\delta)^e}$. We further let $\tilde{G} \triangleq (D^*+G')$. Then we have $\mathbb{E}[\|\x^{t+1}-\x^*\|^2_\Phi] \le (D^*+G')\tilde{\eta}^{t+1}$. Thus, $\{\x^t\}$ converges linearly in an expected-value sense. \\ 
		(b) Case (i): If $\delta<\eta<1$. From (a), it follows that
		\begin{align*}
			\mathbb{E}[\|\x^{T+1}-\x^*\|^2_\Phi] &\le \tilde{G}{\eta}^{T+1}\leq \  \epsilon \Longrightarrow  T \ge  \log_{1/{\eta}}(\tilde{G}/\epsilon) - 1 .
		\end{align*} 
		If $T = \lceil \log_{1/{\eta}}(\tilde{G}/\epsilon)\rceil - 1$, then ({\bf DVRSFBF}) requires $\sum_{t=0}^TS_t+2KT$ evaluations. Since $S_t=\lfloor \eta^{-(t+1)} \rfloor \le \eta^{-(t+1)}$, then we have 
		{\begin{align*}
			\quad \sum_{t=0}^{\lceil \log_{1/{\eta} }(\tilde{G}/\epsilon)\rceil -1}&\eta^{-(t+1)}  =
			\sum_{s=1}^{\lceil \log_{1/{\eta} }(\tilde{G}/\epsilon)\rceil}\eta^{-s}\\ 
			&\le \tfrac{1}{\eta\left(\tfrac{1}{\eta}-1\right)}\left(\tfrac{1}{\eta}\right)^{\lceil \log_{1/{\eta}}(\tilde{G}/\epsilon)\rceil}\\ 
			&\le \tfrac{1}{\left(1-\eta\right)}\left(\tfrac{1}{\eta}\right)^{\log_{1/\eta}(\tilde{G}/\epsilon)}
			\le \tfrac{1}{(1-\eta)}\left(\tfrac{\tilde{G}}{\epsilon}\right).
		\end{align*}
		It follows that
		\begin{align*}
			\sum_{t=0}^TS_t+2KT \le \tfrac{1}{(1-\eta)}\left(\tfrac{\tilde{G}}{\epsilon}\right)+2K\log_{1/{\eta}}(\tilde{G}/\epsilon).
		\end{align*}}
		Since $\delta<\eta<1$, we see that $q^{K}+\frac{1-q^{K}}{1-q}\rho<\eta$ and thus, we derive $K>\log_q\frac{\eta(1-q)-\rho}{1-q-\rho}$. Noting that $\tilde G \triangleq (D^*+\tfrac{\bar G}{1-\delta/\eta})=\tilde{C}\frac{1-q^K}{\frac{\eta(1-q)-\rho}{1-q-\rho}-q^K}$ for some appropriately defined $\tilde{C}$. We conclude $K= \lfloor \log_q\frac{\eta(1-q)-\rho}{1-q-\rho} \rfloor +1$ is the optimal choice. \\
		We omit cases (ii) and (iii) which lead to similar complexities.
	\end{proof}
	
\subsection{Convergence analysis under monotonicity}
    \begin{lemma}\label{lemma:bounded}
        Let Assumptions~\ref{ass:convex} and \ref{ass:X} hold. Then the solution set of the primal-dual problem \eqref{KKT_VI} is bounded. 
    \end{lemma}
    \begin{proof}
        Since Assumption~\ref{ass:convex} holds, the operator \( F: \mathbb{R}^d \rightarrow \mathbb{R}^d \) is (maximally) monotone. By Assumption~\ref{ass:X}, the solution set of the primal problem \eqref{eq:VI} is bounded. Hence, according to \cite[Proposition 3.3]{auslender2000}, the solution set of the associated dual problem is nonempty and bounded. Consequently, the solution set of the primal-dual formulation \eqref{KKT_VI} is bounded.
    \end{proof}
    
    By Lemma~\ref{lemma:bounded}, we can assume that \(\zer(\opV+\opT)\) is bounded. Based on this, the convergence behavior can be analyzed within a neighborhood of an arbitrary center point \( \x_c\in\setX \) by the \emph{restricted merit function} \cite{Nes07} that reads as
    \begin{align*}
        \text{Gap}(\z) = \sup_{\x \, \in \, \setX}\{&\inner{\opV(\x),\z -\x} + H(\z) - H(\x):\\
        &\norm{\x_c-\x}^2\leq C^2\},
    \end{align*}
    where $C$ is a constant and $H(\z):=g(\ubold)+\delta_{\R^{mn}_{\geq 0}}(y), \forall \z=(\ubold,p,y)\in\setX$. The following proposition formalizes this result.
    
 	\begin{proposition} \em Let Assumptions~\ref{ass:convex}-\ref{moment_biased} hold. Suppose the conditional means $\mathbb{E}[\eps_{j,t}(\x^t)  \mid  {\cal F}_t ],\, \mathbb{E}[\eps_{k+\frac{1}{2},t}(\z^t_{k+\frac{1}{2}}) \mid {\cal F}_{k,t} ],\, \mathbb{E}[\eps_{k+\frac{1}{2},t}(\x^t) \mid {\cal F}_{k,t}]$ are zero for all $t,k$ and $j$ in an a.s. sense. Let ${\setX}$ be a closed, convex, and nonempty set in $\R^n$. Suppose \textnormal{Zer}$(\mathbb{V}+\mathbb{T})$ is nonempty. Given $T > 0$, consider the sequence generated by ({\bf DVRSFBF}) where $\lambda_{\max}(\Phi^{-1}) = 1/T$, $K = T$, and $S_t \geq 1/\lambda^2_{\max}(\Phi^{-1})$. Suppose $\bar{\z}^T \, \triangleq \, \sum_{t=0}^{T-1} \bar{\z}^t/T$ and $\bar{\z}^t$ represents the averaged iterate at the $t$-th epoch, defined as $\bar{\z}^t \, \triangleq  \, \sum_{k=0}^{K-1} \z^t_{k+\frac{1}{2}}/K$. {Suppose $\x_c\in \setX$ is such a point that $\norm{\z^t_{\tilde{k}/2}-\x_c}^2\leq C^2$ for $t=0,1,\dots,T$ and $\tilde{k}=0,1,\dots,2K$, where $C$ is a constant. 
    Then the following hold.}
    \begin{itemize}
        \item [(a)] For any $T > 0$, we have that 
		\begin{align}
			\mathbb{E}\left[ \, \textnormal{Gap}(\bar{\z}^T) \, \right] \, \leq \, {\cal O}\left(\tfrac{1}{T}\right).
		\end{align}
        \item [(b)] Let $S_t=T^2$. Suppose $\mathbb{E}\left[ \, \op{Gap}(\bar{\z}^T) \, \right]  \leq\epsilon$. Then $\sum_{t=0}^TS_t + 2KT \le
		\mathcal{O}\left(\frac{1}{\epsilon^3}\right).$
    \end{itemize}    
    \end{proposition}
    \begin{proof}
	{
        We recall the following iteration 
	\begin{align*}
		\z^t_{k+\frac{1}{2}}&= \mathrm{J}_{\Phi^{-1}\opT}(\z^t_k-\Phi^{-1}{\bar{\opV}}(\x^t)), \\
		\z^t_{k+1}&= \z^t_{k+\frac{1}{2}}-\Phi^{-1}({\tilde{\opV}}(\z^t_{k+\frac{1}{2}},\xi_{k+\frac{1}{2},t})-{\tilde{\opV}}(\x^t,\xi_{k+\frac{1}{2},t}))
	\end{align*}
	To deduce the necessary energy properties of the scheme, we consider the optimality condition of the first iterate:
	$$
	\Phi(\z^{t}_{k}-\z^t_{k+\frac{1}{2}})-\bar{\opV}(\x^t)\in \opT(\z^t_{k+\frac{1}{2}})
	$$
	For $g=\sum_{i\in\scrI}g_{i}$, this optimality condition reads in coordinates as 
	\begin{equation}
		g(\ubold)-g(\ubold^{t}_{k+\frac{1}{2}}) \geq \inner{\z^{t}_{k}-\z^t_{k+\frac{1}{2}}-\Phi^{-1}\bar{\opV}(\x^{t}),\x-\z^t_{k+\frac{1}{2}}}_{\Phi} 
	\end{equation}
	Fix $\x=(\ubold,p,y)\in\setM\triangleq\{\z\in\setX \mid \norm{\z-\x_c}^2\leq C^2\}$ as reference point. We then compute 
	\begin{align*}
		&\norm{\z^{t}_{k+1}-\x}^{2}_{\Phi}=\norm{\z^{t}_{k+1}-\z^t_{k+\frac{1}{2}}}^{2}_{\Phi}-\norm{\z^t_{k+\frac{1}{2}}-\z^{t}_{k}}^{2}_{\Phi}\\
		&~~+\norm{\z^{t}_{k}-\x}^{2}_{\Phi}+2\inner{\z^{t}_{k+1}-\z^{t}_{k},\z^t_{k+\frac{1}{2}}-\x}_{\Phi}\\
		&~=\norm{\z^{t}_{k}-\x}^{2}_{\Phi}+\norm{\tilde{\opV}(\z^{t}_{k+1},\xi^{t}_{k+\frac{1}{2}})-\tilde{\opV}(\x^{t},\xi^{t}_{k+\frac{1}{2}})}_{\Phi^{-1}}^{2}\\
		&~~-\norm{\z^t_{k+\frac{1}{2}}-\z^{t}_{k}}^{2}_{\Phi}+2\inner{\z^{t}_{k+1}-\z^t_{k+\frac{1}{2}},\z^t_{k+\frac{1}{2}}-\x}_{\Phi}\\
            &~~+2\inner{\z^t_{k+\frac{1}{2}}-\z^{t}_{k},\z^t_{k+\frac{1}{2}}-\x}_{\Phi}\\
		&~=\norm{\z^{t}_{k}-\x}^{2}_{\Phi}+\norm{\tilde{\opV}(\z^{t}_{k+1},\xi^{t}_{k+\frac{1}{2}})-\tilde{\opV}(\x^{t},\xi^{t}_{k+\frac{1}{2}})}_{\Phi^{-1}}^{2}\\
		&~~-\norm{\z^t_{k+\frac{1}{2}}-\z^{t}_{k}}^{2}_{\Phi}+2\inner{\z^{t}_{k+1}-\z^t_{k+\frac{1}{2}},\z^t_{k+\frac{1}{2}}-\x}_{\Phi}\\
		&~~+2\inner{\z^t_{k+\frac{1}{2}}-(\z^{t}_{k}-\Phi^{-1}\bar{\opV}(\x^{t})),\z^t_{k+\frac{1}{2}}-\x}_{\Phi}\\
            &~~+2\inner{\bar{\opV}(\x^{t}),\x-\z^t_{k+\frac{1}{2}}}\\
		&~\leq \norm{\z^{t}_{k}-\x}^{2}_{\Phi}+\norm{\tilde{\opV}(\z^{t}_{k+1},\xi^{t}_{k+\frac{1}{2}})-\tilde{\opV}(\x^{t},\xi^{t}_{k+\frac{1}{2}})}_{\Phi^{-1}}^{2}\\
		&~~-\norm{\z^t_{k+\frac{1}{2}}-\z^{t}_{k}}^{2}_{\Phi}+2\inner{\z^{t}_{k+1}-\z^t_{k+\frac{1}{2}},\z^t_{k+\frac{1}{2}}-\x}_{\Phi}\\
		  &~~-2\inner{\bar{\opV}(\x^{t}),\z^t_{k+\frac{1}{2}}-\x} + 2[g(\ubold)-g(\ubold^{t}_{k+\frac{1}{2}})]\\
            &~\leq \norm{\z^{t}_{k}-\x}^{2}_{\Phi}+\norm{\tilde{\opV}(\z^{t}_{k+1},\xi^{t}_{k+\frac{1}{2}})-\tilde{\opV}(\x^{t},\xi^{t}_{k+\frac{1}{2}})}_{\Phi^{-1}}^{2}\\
		&~~-\norm{\z^t_{k+\frac{1}{2}}-\z^{t}_{k}}^{2}_{\Phi} -2\inner{\bar{\opV}(\x^{t}),\z^t_{k+\frac{1}{2}}-\x}\\
            &~~-2\inner{\tilde{\opV}(\z^t_{k+\frac{1}{2}}),\xi_{k+\frac{1}{2},t})-\tilde{\opV}(\x^{t},\xi_{k+\frac{1}{2},t}),\z^t_{k+\frac{1}{2}}-\x}\\
            &~~ + 2[g(\ubold)-g(\ubold^{t}_{k+\frac{1}{2}})].\\
	\end{align*}
	To make this expression appear a bit more symmetric, we can define the function $H:\setX\to\R$ by 
	$$
	H(\z):=g(\ubold)+\delta_{\R^{mn}_{\geq 0}}(y)\qquad\forall \z=(\ubold,p,y)\in\setX.
	$$
	We note in passing that $\partial H(\z)=\opT(\z)$. Using now the Lipschitz continuity of the pseudo-gradient, we can continue the above estimate by 
	\begin{align*}
		&\norm{\z^{t}_{k+1}-\x}^{2}_{\Phi}\leq\norm{\z^{t}_{k}-\x}^{2}_{\Phi} -\norm{\z^t_{k+\frac{1}{2}}-\z^{t}_{k}}^{2}_{\Phi}\\
            &~+2\inner{\bar{W}^{t},\z^t_{k+\frac{1}{2}}-\x}-2\inner{\opV(\x^{t}),\z^t_{k+\frac{1}{2}}-\x}\\
		&~+[L_{\norm{\cdot}_{\Phi^{-1}}}(\xi_{k+\frac{1}{2},t})]^{2}\norm{\z^t_{k+\frac{1}{2}}-\x^{t}}_{\Phi}^{2}\\
            &~-2\inner{\tilde{\opV}(\z^t_{k+\frac{1}{2}},\xi_{k+\frac{1}{2},t})-\tilde{\opV}(\x^{t},\xi_{k+\frac{1}{2},t}),\z^t_{k+\frac{1}{2}}-\x}\\
		&~+2[H(\x)-H(\z^t_{k+\frac{1}{2}})],
	\end{align*}
	where we have set $\bar{W}^{t}+\bar{\opV}(\x^{t})=\opV(\x^{t})$. Recall the definition in \eqref{eq:def_eps} and observe that 
	\begin{align*}
		2&\inner{\tilde{\opV}(\z^t_{k+\frac{1}{2}},\xi_{k+\frac{1}{2},t})-\tilde{\opV}(\x^{t},\xi_{k+\frac{1}{2},t}),\z^t_{k+\frac{1}{2}}-x}\\
            &=2\inner{\eps_{k+\frac{1}{2},t}(\z^t_{k+\frac{1}{2}})-\eps_{k+\frac{1}{2},t}(\x^{t}),\z^t_{k+\frac{1}{2}}-\x}\\
		&~+\underbrace{2\inner{\opV(\z^t_{k+\frac{1}{2}})-\opV(\x),\z^t_{k+\frac{1}{2}}-\x}}_{\geq0}\\
            &~-2\inner{\opV(\x^t)-\opV(\x),\z^t_{k+\frac{1}{2}}-\x}\\
		&\geq -2\inner{\opV(\x^t)-\opV(\x),\z^t_{k+\frac{1}{2}}-\x}\\
            &~+2\inner{\eps_{k+\frac{1}{2},t}(\z^t_{k+\frac{1}{2}})-\eps_{k+\frac{1}{2},t}(\x^{t}),\z^t_{k+\frac{1}{2}}-\x}.
	\end{align*}
	Using this bound, and rearranging the previous estimates, we arrive at 
	\begin{align*}
		2&\inner{\opV(\x),\z^t_{k+\frac{1}{2}}-\x}+2[H(\z^t_{k+\frac{1}{2}})-H(\x)] \\
            & \leq \norm{\z^{t}_{k}-\x}^{2}_{\Phi}-\norm{\z^{t}_{k+1}-\x}^{2}_{\Phi} +2\inner{\bar{W}^{t},\z^t_{k+\frac{1}{2}}-\x}\\
		& ~+[L_{\norm{\cdot}_{\Phi^{-1}}}(\xi_{k+\frac{1}{2},t})]^{2}\norm{\z^t_{k+\frac{1}{2}}-\x^{t}}_{\Phi}^{2}-\norm{\z^t_{k+\frac{1}{2}}-\z^{t}_{k}}_{\Phi}^{2} \\
            &~-2\inner{\eps_{k+\frac{1}{2},t}(\z^t_{k+\frac{1}{2}})-\eps_{k+\frac{1}{2},t}(\x^{t}),\z^t_{k+\frac{1}{2}}-\x}.
	\end{align*}
	Defining the function 
	\begin{equation}
		\scrQ(\x,\z):=\inner{\opV(\x),\z-\x}+H(\z)-H(\x), 
	\end{equation}
	we can continue the above estimate as 
	\begin{align*}
		2&\scrQ(\x,\z^t_{k+\frac{1}{2}}) \leq \norm{\z^{t}_{k}-\x}^{2}_{\Phi}-\norm{\z^{t}_{k+1}-\x}^{2}_{\Phi} -\norm{\z^t_{k+\frac{1}{2}}-\z^{t}_{k}}_{\Phi}^{2}\\
		& +[L_{\norm{\cdot}_{\Phi^{-1}}}(\xi_{k+\frac{1}{2},t})]^{2}\norm{\z^t_{k+\frac{1}{2}}-\x^{t}}_{\Phi}^{2} +2\inner{\bar{W}^{t},\z^t_{k+\frac{1}{2}}-\x}\\
		&-2\inner{\eps_{k+\frac{1}{2},t}(\z^t_{k+\frac{1}{2}})-\eps_{k+\frac{1}{2},t}(\x^{t}),\z^t_{k+\frac{1}{2}}-\x}.
	\end{align*}
	Dropping the cross term and summing over the inner iteration index $k=0,\ldots,K-1$, this implies 
	\begin{align*}
		&\frac{1}{K}\sum_{k=0}^{K-1}\scrQ(\x,\z^t_{k+\frac{1}{2}}) \leq \frac{1}{2K}\norm{\z^{t}_{0}-\x}^{2}_{\Phi}-\frac{1}{2K}\norm{\z^{t}_{K}-\x}^{2}_{\Phi}\\
		&+\frac{1}{2K}\sum_{k=0}^{K-1}[L_{\norm{\cdot}_{\Phi^{-1}}}(\xi_{k+\frac{1}{2},t})]^{2}\norm{\z^t_{k+\frac{1}{2}}-\x^{t}}_{\Phi}^{2} +\inner{\bar{W}^{t},\bar{\z}^{t}-x} \\
		&-\frac{1}{K}\sum_{k=0}^{K-1}\inner{\eps_{k+\frac{1}{2},t}(\z^t_{k+\frac{1}{2}})-\eps_{k+\frac{1}{2},t}(\x^{t}),\z^t_{k+\frac{1}{2}}-\x}
	\end{align*}
	where $\bar{\z}^{t}:=\frac{1}{K}\sum_{k=0}^{K-1}\z^t_{k+\frac{1}{2}}$. Since $x\mapsto H(x)$ is convex, we have 
	$$
	\frac{1}{K}\sum_{k=0}^{K-1}\scrQ(\x,\z^t_{k+\frac{1}{2}})\geq \scrQ(\x,\bar{\z}^{t}). 
	$$
	Furthermore, $\z^{t}_{0}=\x^{t}$ and $\x^{t+1}=\z^{t}_{K}$, so that 
	\begin{align*}
		&\scrQ(\x,\bar{\z}^{t})\leq \frac{1}{2K}\norm{\x^{t}-\x}^{2}_{\Phi}-\frac{1}{2K}\norm{\x^{t+1}-\x}^{2}_{\Phi}\\
		&+\frac{1}{2K}\sum_{k=0}^{K-1}[L_{\norm{\cdot}_{\Phi^{-1}}}(\xi_{k+\frac{1}{2},t})]^{2}\norm{\z^t_{k+\frac{1}{2}}-\x^{t}}_{\Phi}^{2} +\inner{\bar{W}^{t},\bar{\z}^{t}-\x}\\
		&-\frac{1}{K}\sum_{k=0}^{K-1}\inner{\eps_{k+\frac{1}{2},t}(\z^t_{k+\frac{1}{2}})-\eps_{k+\frac{1}{2},t}(\x^{t}),\z^t_{k+\frac{1}{2}}-\x}
	\end{align*}
	Following the classical analysis of online stochastic approximation methods, we introduce auxiliary processes $(A^{t}_{k})_{k}$ and $(B^{t}_{k})_{k}$ by 
	$$
	A^{t}_{K}=A^{t+1}_{0},B^{t}_{K}=B^{t+1}_{0}, 
	$$
	and 
	\begin{align*}
		A^{t}_{k+1}&=P^{\Phi}_{\setX}[A^{t}_{k}+\Phi^{-1}\eps_{k+\frac{1}{2},t}(\z^t_{k+\frac{1}{2}})],\\ 
		B^{t}_{k+1}&=P^{\Phi}_{\setX}[B^{t}_{k}-\Phi^{-1}\eps_{k+\frac{1}{2},t}(\x^{t})],
	\end{align*}
	where $P^{\Phi}_{\setX}(x)=\op{argmin}_{x'\in\setX}\frac{1}{2}\norm{x-x'}_{\Phi}^{2}$ is the orthogonal projection onto $\setX$ under the weighed metric $\norm{\cdot}_{\Phi}$. From the definitions of these processes, we obtain 
	\begin{align*}
		\norm{A^{t}_{k+1}-\x}^{2}_{\Phi}\leq &\norm{A^{t}_{k}-\x}^{2}_{\Phi}+2\inner{\eps_{k+\frac{1}{2},t}(\z^t_{k+\frac{1}{2}}),A^{t}_{k}-\x}\\
            &+\norm{\eps_{k+\frac{1}{2},t}(\z^t_{k+\frac{1}{2}})}^{2}_{\Phi^{-1}},\\ 
		\norm{B^{t}_{k+1}-\x}^{2}_{\Phi}\leq &\norm{B^{t}_{k}-\x}^{2}_{\Phi}-2\inner{\eps_{k+\frac{1}{2},t}(\x^{t}),B^{t}_{k}-\x}\\
            &+\norm{\eps_{k+\frac{1}{2},t}(\x^{t})}^{2}_{\Phi^{-1}}.
	\end{align*}
	Hence, 
	\begin{align*}
		&\sum_{k=0}^{K-1}\inner{\eps_{k+\frac{1}{2},t}(\z^t_{k+\frac{1}{2}}),\x-\z^t_{k+\frac{1}{2}}}\leq \frac{1}{2}\norm{A^{t}_{0}-x}^{2}_{\Phi}\\
            &~-\frac{1}{2}\norm{A^{t+1}_{0}-x}^{2}_{\Phi} +\sum_{k=0}^{K-1}\inner{\eps_{k+\frac{1}{2},t}(\z^t_{k+\frac{1}{2}}),A^{t}_{k}-\z^t_{k+\frac{1}{2}}}\\
		&~+\frac{1}{2}\sum_{k=0}^{K-1}\norm{\eps_{k+\frac{1}{2},t}(\z^t_{k+\frac{1}{2}})}^{2}_{\Phi^{-1}}.
	\end{align*}
	In exactly the same way, we obtain 
	\begin{align*}
		&\sum_{k=0}^{K-1}\inner{\eps_{k+\frac{1}{2},t}(\x^{t}),\z^t_{k+\frac{1}{2}}-\x}\leq \frac{1}{2}\norm{B^{t}_{0}-\x}^{2}_{\Phi}\\
            &~-\frac{1}{2}\norm{B^{t+1}_{0}-\x}^{2}_{\Phi}+\sum_{k=0}^{K-1}\inner{\eps_{k+\frac{1}{2},t}(\x^{t}),\z^t_{k+\frac{1}{2}}-B^{t}_{k}}\\
		&~+\frac{1}{2}\sum_{k=0}^{K-1}\norm{\eps_{k+\frac{1}{2},t}(\x^{t})}^{2}_{\Phi^{-1}}.
	\end{align*}
	Substituting this into our upper bound for the gap function, we continue our estimation with 
	\begin{align*}
		\scrQ&(\x,\bar{\z}^{t})\leq\frac{1}{2K}\left(\norm{\x^{t}-\x}^{2}_{\Phi}-\norm{\x^{t+1}-\x}^{2}_{\Phi}\right)\\
            &+\frac{1}{2K}\sum_{k=0}^{K-1}[L_{\norm{\cdot}_{\Phi^{-1}}}(\xi_{k+\frac{1}{2},t})]^{2}\norm{\z^t_{k+\frac{1}{2}}-\x^{t}}_{\Phi}^{2} \\
		&+\frac{1}{2K}\left(\norm{A^{t}_{0}-\x}^{2}_{\Phi}-\norm{A^{t+1}_{0}-\x}^{2}_{\Phi}\right)\\
            &+\frac{1}{2K}\left(\norm{B^{t}_{0}-\x}^{2}_{\Phi}-\norm{B^{t+1}_{0}-\x}^{2}_{\Phi}\right) +\inner{\bar{W}^{t},\bar{\z}^{t}-\x}\\
		&+\frac{1}{2K}\sum_{k=0}^{K-1}\left(\norm{\eps_{k+\frac{1}{2},t}(\z^t_{k+\frac{1}{2}})}^{2}_{\Phi^{-1}}+\norm{\eps_{k+\frac{1}{2},t}(\x^{t})}^{2}_{\Phi^{-1}}\right)\\
		&+\frac{1}{K}\sum_{k=0}^{K-1}\inner{\eps_{k+\frac{1}{2},t}(\z^t_{k+\frac{1}{2}}),A^{t}_{k}-\z^t_{k+\frac{1}{2}}}\\
            &+\frac{1}{K}\sum_{k=0}^{K-1}\inner{\eps_{k+\frac{1}{2},t}(\x^{t}),\z^t_{k+\frac{1}{2}}-B^{t}_{k}}.
	\end{align*}
	Next, we sum over the outer iteration index counter $t=0,1,\ldots,T-1$ and call $\bar{\z}_{T}=\frac{1}{T}\sum_{t=0}^{T-1}\bar{\z}^{t}$, to get 
	\begin{align*}
		&\scrQ(\x,\bar{\z}_{T})\leq \frac{1}{2KT}\left(\norm{\x^{0}-\x}^{2}_{\Phi}-\norm{\x^{T}-\x}^{2}_{\Phi}\right)\\
            &+\frac{1}{2KT}\sum_{t=0}^{T-1}\sum_{k=0}^{K-1}[L_{\norm{\cdot}_{\Phi^{-1}}}(\xi_{k+\frac{1}{2},t})]^{2}\norm{\z^t_{k+\frac{1}{2}}-\x^{t}}_{\Phi}^{2} \\
		&+\frac{1}{2KT}\left(\norm{A^{0}_{0}-\x}^{2}_{\Phi}-\norm{A^{T}_{0}-\x}^{2}_{\Phi}\right)\\
            &+\frac{1}{2KT}\left(\norm{B^{0}_{0}-\x}^{2}_{\Phi}-\norm{B^{T}_{0}-\x}^{2}_{\Phi}\right)+\frac{1}{T}\sum_{t=0}^{T-1}\inner{\bar{W}^{t},\bar{\z}^{t}-\x}\\
		&+\frac{1}{2KT}\sum_{t=0}^{T-1}\sum_{k=0}^{K-1}\left(\norm{\eps_{k+\frac{1}{2},t}(\z^t_{k+\frac{1}{2}})}^{2}_{\Phi^{-1}}+\norm{\eps_{k+\frac{1}{2},t}(\x^{t})}^{2}_{\Phi^{-1}}\right)\\
		&+\frac{1}{KT}\sum_{t=0}^{T-1}\sum_{k=0}^{K-1}\inner{\eps_{k+\frac{1}{2},t}(\z^t_{k+\frac{1}{2}}),A^{t}_{k}-\z^t_{k+\frac{1}{2}}}\\
            &+\frac{1}{KT}\sum_{t=0}^{T-1}\sum_{k=0}^{K-1}\inner{\eps_{k+\frac{1}{2},t}(\x^{t}),\z^t_{k+\frac{1}{2}}-B^{t}_{k}}.
	\end{align*}
	To perform the final bounds, we set $A^{0}_{0}=x^{0}=B^{0}_{0}$, and ignore the negative terms, we obtain 
	\begin{align*}
		&\sup_{\x\in\setM}\scrQ(\x,\bar{z}_{T})\leq \sup_{\x\in\setM}\frac{3}{2KT}\norm{\x^{0}-\x}^{2}_{\Phi}+\frac{1}{T}\sum_{t=0}^{T-1}\inner{\bar{W}^{t},\bar{\z}^{t}-\x}\\
            &~+\frac{1}{2KT}\sum_{t=0}^{T-1}\sum_{k=0}^{K-1}[L_{\norm{\cdot}_{\Phi^{-1}}}(\xi_{k+\frac{1}{2},t})]^{2}\norm{\z^t_{k+\frac{1}{2}}-\x^{t}}_{\Phi}^{2} \\
            &~+\frac{1}{2KT}\sum_{t=0}^{T-1}\sum_{k=0}^{K-1}\left(\norm{\eps_{k+\frac{1}{2},t}(\z^t_{k+\frac{1}{2}})}^{2}_{\Phi^{-1}}+\norm{\eps_{k+\frac{1}{2},t}(\x^{t})}^{2}_{\Phi^{-1}}\right)\\
		&~+\frac{1}{KT}\sum_{t=0}^{T-1}\sum_{k=0}^{K-1}\inner{\eps_{k+\frac{1}{2},t}(\z^t_{k+\frac{1}{2}}),A^{t}_{k}-\z^t_{k+\frac{1}{2}}}\\
            &~+\frac{1}{KT}\sum_{t=0}^{T-1}\sum_{k=0}^{K-1}\inner{\eps_{k+\frac{1}{2},t}(\x^{t}),\z^t_{k+\frac{1}{2}}-B^{t}_{k}}
	\end{align*}
	where $\setM\triangleq\{\z\in\setX \mid \norm{\z-\x_c}^2\leq C^2\}$ is a compact set containing $\x$ in its interior. By invoking the definition of the restricted merit function at $\bar{\z}^T$, i.e. $\text{Gap}(\z) = \sup_{\x \, \in \, \setX}\{\inner{\opV(\x),\z -\x} + H(\z) - H(\x):\norm{\x_c-\x}^2\leq C^2\}$ and then taking expectations on both sides, we obtain    
		\begin{align}
			\notag&\mathbb{E}\left[ \,  \text{Gap}(\bar{\z}^T) \,  \right] \leq \frac{3}{2KT}\cdot\|\Phi\|\cdot 2C^2+\frac{1}{T}\sum_{t=0}^{T-1}\EE \left[\norm{\bar{W}^{t}}\cdot \sqrt{2}C\right]\\
            \notag&~+\frac{1}{2KT}\sum_{t=0}^{T-1}\sum_{k=0}^{K-1}\EE \left[ [L_{\norm{\cdot}_{\Phi^{-1}}}(\xi_{k+\frac{1}{2},t})]^{2}\norm{\z^t_{k+\frac{1}{2}}-\x^{t}}_{\Phi}^{2}\right] \\
            \notag&~+\frac{1}{2KT}\sum_{t=0}^{T-1}\sum_{k=0}^{K-1}\EE \left[\norm{\eps_{k+\frac{1}{2},t}(\z^t_{k+\frac{1}{2}})}^{2}_{\Phi^{-1}}\right]\\
            \notag&~+\frac{1}{2KT}\sum_{t=0}^{T-1}\sum_{k=0}^{K-1}\EE \left[\norm{\eps_{k+\frac{1}{2},t}(\x^{t})}^{2}_{\Phi^{-1}}\right]\\
		\notag&~+\frac{1}{KT}\sum_{t=0}^{T-1}\sum_{k=0}^{K-1}\EE \left[\inner{\eps_{k+\frac{1}{2},t}(\z^t_{k+\frac{1}{2}}),A^{t}_{k}-\z^t_{k+\frac{1}{2}}}\right]
        \end{align}
        \begin{align}
            \notag&~+\frac{1}{KT}\sum_{t=0}^{T-1}\sum_{k=0}^{K-1}\EE \left[\inner{\eps_{k+\frac{1}{2},t}(\x^{t}),\z^t_{k+\frac{1}{2}}-B^{t}_{k}}\right]\\
            \notag&\leq\frac{3}{2KT}\cdot\|\Phi\|\cdot 2C^2 + \frac{1}{2KT}\sum_{t=0}^{T-1}\sum_{k=0}^{K-1}\|\Phi^{-1}\|^2\hat{L}^2\cdot2C^2\cdot\|\Phi\|\\
            \notag&~+\frac{1}{T}\sum_{t=0}^{T-1}\cdot\frac{\sqrt{2}C\nu}{\sqrt{S_t}}+\frac{1}{2KT}\sum_{t=0}^{T-1}\sum_{k=0}^{K-1}\left(\|\Phi^{-1}\|\nu^2+\|\Phi^{-1}\|\nu^2\right)\\
            &\leq\frac{3C^2}{T}+\frac{\hat{L}^2C^2}{T}+\frac{\sqrt{2}C\nu}{T}+\frac{\nu^2}{T}, 
		\end{align}
		where the second inequality follows by imposing conditional unbiasedness and the third inequality is a result of choosing $\lambda_{\max}(\Phi^{-1}) = 1/T$, $K = T$, and $S_t \geq 1/\lambda^2_{\max}(\Phi^{-1})$.  

        (b) Let $Q\triangleq 3C^2+\hat{L}^2C^2+\sqrt{2}C\nu+\nu^2$. We have that
        \begin{align*}
            \mathbb{E}\left[ \, G(\bar{\z}^T) \,\right] \leq\frac{Q}{T}\leq \epsilon \Longrightarrow T\geq \frac{Q}{\epsilon}
        \end{align*}
        Let $T=\lceil\frac{Q}{\epsilon}\rceil\leq\frac{Q}{\epsilon}+1$. Since DVRSFBF requires $\sum_{t=0}^TS_t + 2KT$ evaluations,
        \begin{align*}
            \sum_{t=0}^TS_t + 2KT=(T+1)T^2+2T^2\leq \frac{Q^3}{\epsilon^3} + \frac{6Q^2}{\epsilon^2} + \frac{9Q}{\epsilon} + 4.
        \end{align*}
        }
	\end{proof}
	
\section{Numerical Results}
\begin{figure}[t]\vspace{0.2cm}
	\begin{center}
		\includegraphics[width=6cm]{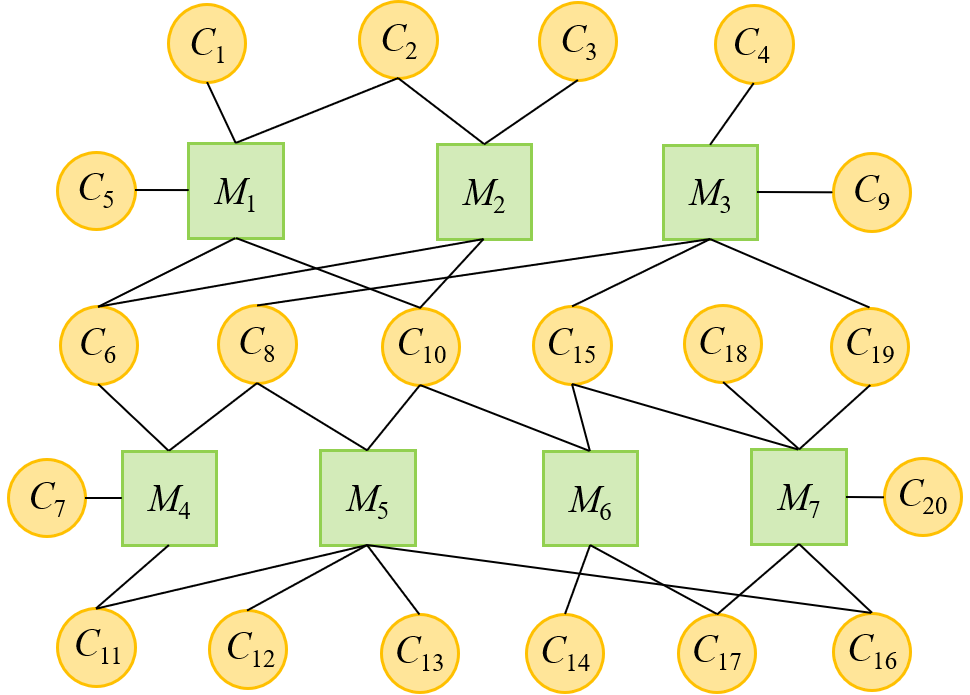}
            \caption{Game setting: the coupling relationship between firms and markets.}\label{fig_market}
	\end{center}	
\end{figure}
\begin{figure}[t]
	\begin{center}
		\includegraphics[width=7cm]{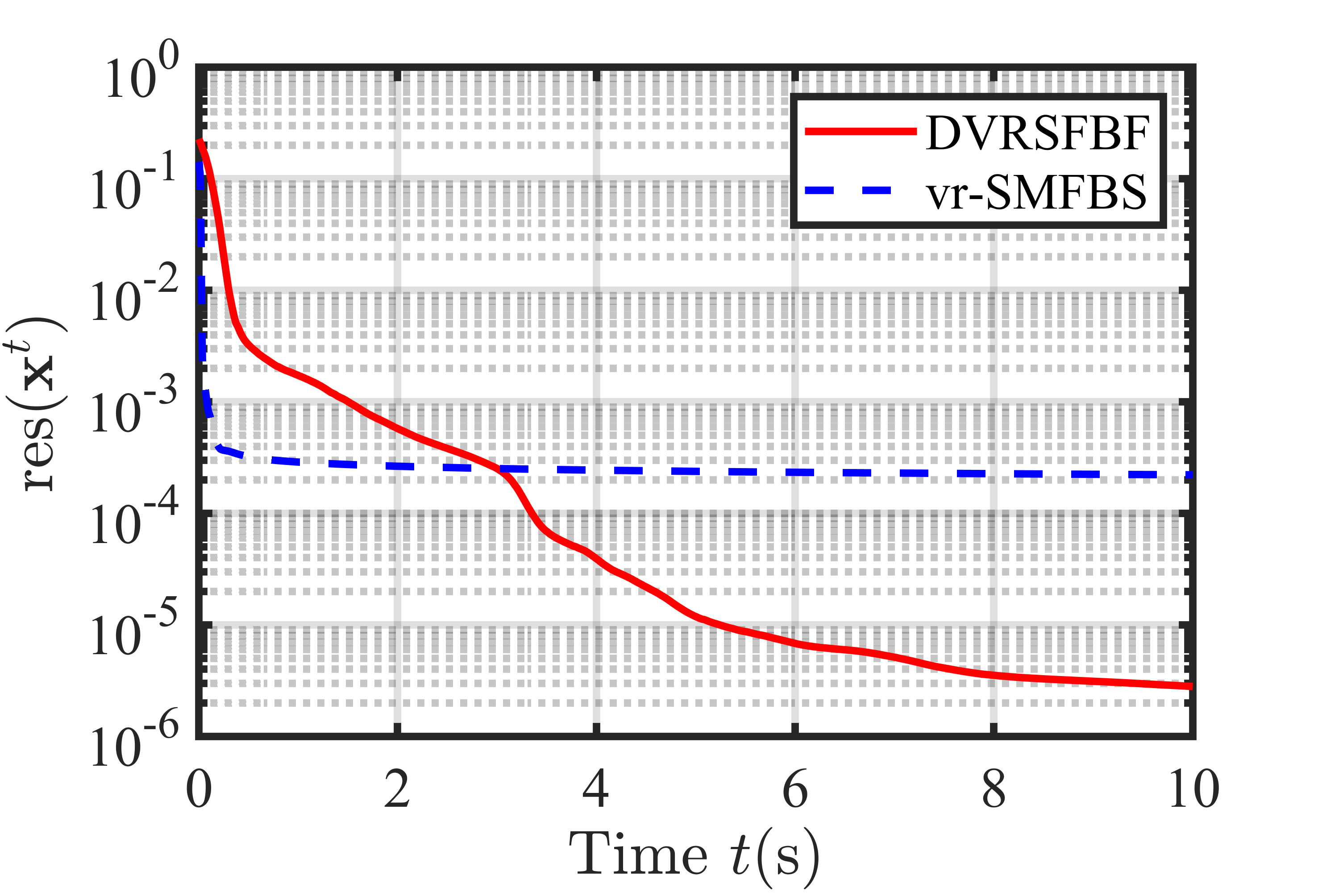}
		\caption{Residual distance of the primal variable form the solution. (strongly mono.)}\label{fig_str}
        \end{center}
\end{figure}
\begin{figure}[t]
	\begin{center}
		\includegraphics[width=7cm]{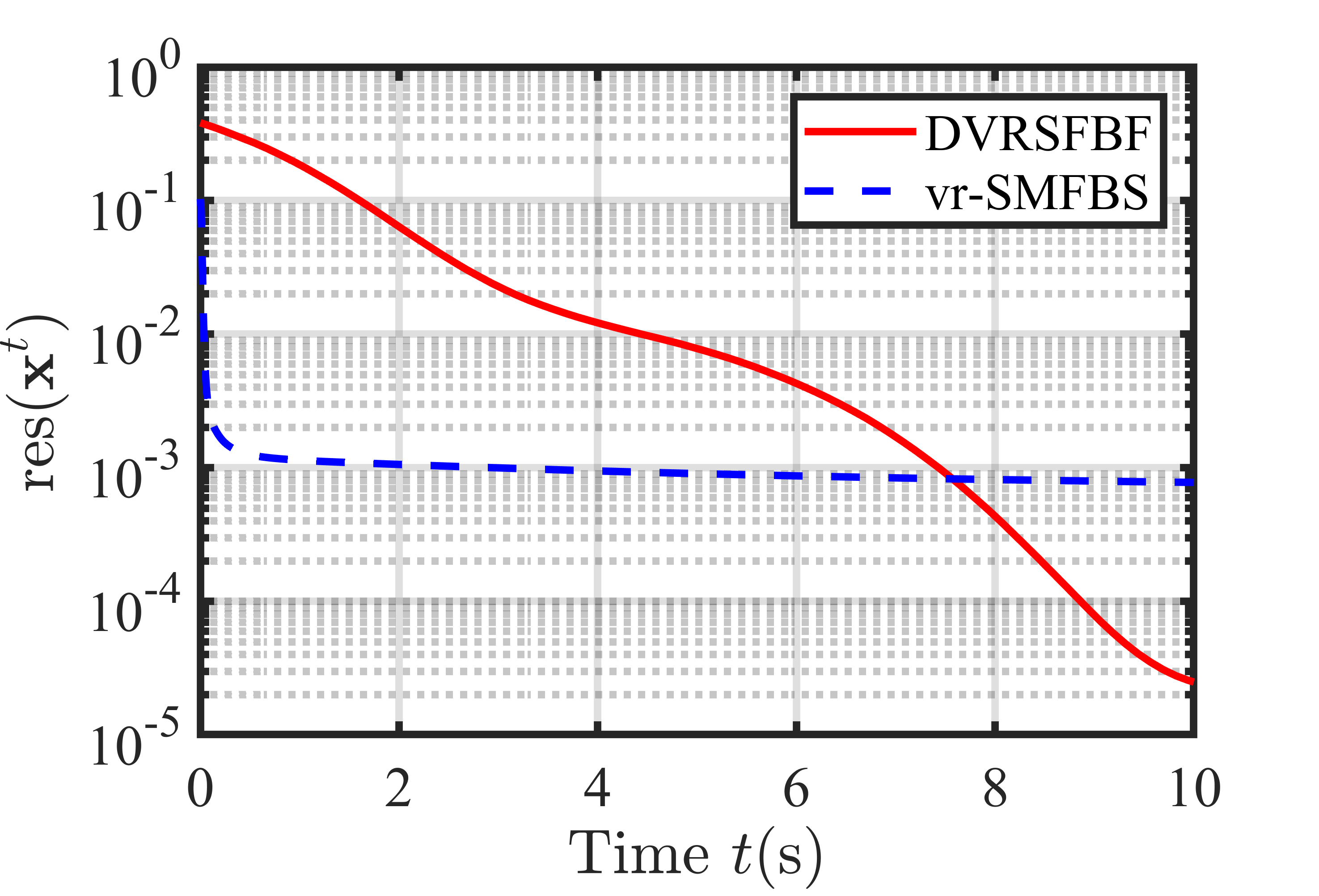}
            \caption{Residual distance of the primal variable form the solution. (merely mono.)}\label{fig_mono}
        \end{center}
\end{figure}
To demonstrate the capabilities of our DVRSFBF algorithm, we conduct numerical simulations in comparison with vr-SMFBS \cite{cui2023variance}. We consider the widely studied networked Cournot game with market capacity constraints and random prices, similar to examples in \cite{yu2017,YiPav19,cui2021relaxed}.

In this game, $N=20$ companies compete in $m=7$ markets, with the coupling relationship depicted in Fig. \ref{fig_market}. Firm $C_i$ sells products to $d_i$ markets, and this correspondence is represented by matrix $A_i \in \R^{m\times d_i}$. By arranging the markets in which firm $i$ participates in ascending order of their indices, if the $k$-th market in which firm $i$ participates is $M_j$, then $[A_i]_{jk} = 1$; otherwise, $[A_i]_{jk} = 0$. Let $A= [A_1\mid\dots\mid A_N]\in\R^{m\times d}$, we can see that the aggregate supply provided by all firms to all markets is $A\ubold = \sum_{i\in \mc I} A_iu_i$. The capacity of all markets are denoted by $b\in \R^m$, with each component randomly selected from the interval $[0.5, 1]$. Thus, we have the global constraint $A\ubold \leq b$. Besides, every firm has its local constraint, i.e., limited production $\0\leq u_i\leq \theta_i$, and each element of $\theta_i$ is randomly drawn from $[1,1.5]$. The production cost of $C_i$ is defined as $c_i(u_i) = a_i(\sum_{j=1}^{d_i}[u_i]_j)^2 + r_i^\top u_i$, where $a_i$ is drawn from $[1,8]$ and each element of $r_i \in \R^{d_i}$ is randomly drawn from $[0.1,0.6]$ for all $i\in \mc I$. In this case, the corresponding pseudogradient mapping is strongly monotone \cite{YiPav19}. Moreover, the inverse demand function for each market $M_j$ is expressed as $P_j(\ubold)=q_j - p_j(\xi)[A\ubold]_j$, where $q_j$ and $p_j$ are randomly drawn from $[2,4]$ and $[5,7]$ respectively. The uncertain parameter $p_j(\xi)$ is normally distributed with a mean of $p_j$ and a bounded variance of 0.1. Finally, we obtain the cost function of $C_i$, i.e., $J_i(u_i,u_{-i}) = c_i(u_i) - \sum_{j\in \mc M}\EE[P_j(\ubold)[A_iu_i]_j]$.

All step size parameters are chosen according to the optimal values discussed in Section \ref{sec:convergence}. Under possible biased estimate settings, the uncertain parameter $p_j(\xi)$ is subject to a normal distribution with a bounded variance of 0.1 and at each iteration step, its mean value is randomly generated around the constant value $p_j$ within an error sphere of radius $1/\sqrt{S_t}$. The biased estimator scenario is restricted to the strongly monotone case. We set $S_t=\lfloor {\eta^{-2(t+1)}} \rfloor$ as in Prop. \ref{ratePvs1} with $\eta=0.99$. In this case, the optimal number of inner iterations \( K=20\). Moreover, we adopt a cycle graph arranged in alphabetical order as the operator graph $\mathcal{G}_y$.

Fig. \ref{fig_str} and \ref{fig_mono} show the empirical behavior averaged over 10 trajectories of DVRSFBF under strong monotonicity and mere monotonicity respectively. The residual of the primal variable $\op{res}(\x^t)=\norm{\op{proj}_{\setC}(\x^t-\Phi^{-1}\opV(\x^t))-\x^t}$ reflects the deviation from the desired VE. We compare our algorithm with variance-reduced stochastic modified forward-backward-forward splitting (vr-SMFBS) \cite{cui2023variance}, which is based on a distributed SFBF scheme with an increasing batch size described as follow:
\begin{equation}\tag{vr-SMFBS}
	\left\{\begin{array}{l}
		\x^{t+\frac{1}{2}}=\mathrm{J}_{\Phi^{-1}\opT}(\x^{t}-\Phi^{-1}\bar{\opV}(\x^{t},\xi_{t})),\\
		\x^{t+1}=\x^{t+\frac{1}{2}}-\Phi^{-1}(\bar{\opV}(\x^{t+\frac{1}{2}},\xi_{t+\frac{1}{2}})-\bar{\opV}(\x^{t},\xi_{t})),
	\end{array}\right.
\end{equation}
where $\bar{V}(\x^t,\xi_t)=\Sigma_{j=1}^{S_t}\tilde{\opV}(\x^t,\xi_{j,t})/{S_t}$, $\bar{\opV}(\x^{t+1/2},\xi_{t+1/2})=\Sigma_{j=1}^{S_t}\tilde{V}(\x^{t+1/2},\xi_{j,t+1/2})/{S_t}$.
The empirical error of DVRSFBF is markedly smaller compared to vr-SMFBS. Initially, due to the limited sample size, the performance of the dual-loop structure of DVRSFBF slightly trails behind vr-SMFBS. However, over time and with an increase in the accuracy of gradient estimation, the superiority of DVRSFBF becomes increasingly pronounced. While the performance of vr-SMFBS slows down significantly after a few seconds, DVRSFBF continues to decline at a relatively fast pace. 

We also conduct comparisons by varying the network scale and batch size, as presented in Table \ref{tab.str} and Table \ref{tab.mono}. The numbers of oracles required to achieve the specified error $\epsilon$ are listed, which show that DVRSFBF requires significantly fewer oracles. These results strongly demonstrate that the proposed algorithm effectively reduces the sampling cost. Table \ref{tab.K} presents the impact of varying the number of inner loop iterations \( K \) on the oracle complexity while keeping all other conditions invariant. Our findings indicate that deviating from the optimal iteration count (\( K = 20 \) in this setting) results in an increase in oracle evaluations. This phenomenon arises because when \( K \) is too small, the SVRG is not fully exploited, necessitating more outer iterations. Conversely, excessively large \( K \) directly increases the total number of iterations and samples.

\begin{table}[t]
	\caption{Oracle complexity required by two algorithms (strongly mono.)}
	\label{tab.str}
	\begin{center}	
		\renewcommand\arraystretch{1.25}
		\begin{tabular}{|l|l|p{40pt}|l|}
			\hline
			\multirow{2}{*}{Size of network} & \multirow{2}{*}{\makecell{Parameter of\\batch size sequence}} & \multicolumn{2}{c|}{Empirical oracle complexity} \\
                \cline{3-4}
                                                 &                                                  & DVRSFBF  & vr-SMFBS \\
			\hline
			\multirow{2}{*}{$N=20,\,m=7$} & $\eta=0.99$ & $6.6\times 10^{5}$ & $ 1.4\times 10^{8}$ \\ 
                \cline{2-4}
                                        & $\eta=0.98$ & $1.8\times 10^{6}$ & $> 10^{9}$\\
			\hline
                \multirow{2}{*}{$N=10,\,m=5$} & $\eta=0.99$ & $1.2\times 10^{5}$ & $1.1\times 10^{6}$ \\ 
                \cline{2-4}
                                        & $\eta=0.98$ & $1.0\times 10^{5}$ & $2.9\times 10^{6}$\\
                \hline
                \multirow{2}{*}{$N=5,\,m=3$} & $\eta=0.99$ & $9.3\times 10^{4}$ & $1.3\times 10^{5}$ \\ 
                \cline{2-4}
                                        & $\eta=0.98$ & $2.3\times 10^{4}$ & $4.6\times 10^{4}$\\
			\hline
                \multicolumn{4}{p{240pt}}{The allowable error $\varepsilon=1\times10^{-4}$, and the batch size $S_t=\lfloor {\eta^{-2(t+1)}} \rfloor$.}
		\end{tabular} 
	\end{center}
\end{table}

\begin{table}[t]
	\caption{Oracle complexity required by two algorithms (merely mono.)}
	\label{tab.mono}
	\begin{center}	
		\renewcommand\arraystretch{1.25}
		\begin{tabular}{|l|l|p{40pt}|l|}
			\hline
			\multirow{2}{*}{Size of network} & \multirow{2}{*}{\makecell{Parameter of\\batch size sequence}} & \multicolumn{2}{c|}{Empirical oracle complexity} \\
                \cline{3-4}
                                                 &                                                  & DVRSFBF  & vr-SMFBS \\
			\hline
			\multirow{2}{*}{$N=20,\,m=7$} & $\alpha=2$ & $3.3\times 10^{6}$ & $> 10^{9}$ \\ 
                \cline{2-4}
                                        & $\alpha=2.5$ & $4.0\times 10^{7}$ & $> 10^{9}$\\
			\hline
                \multirow{2}{*}{$N=10,\,m=5$} & $\alpha=2$ & $5.0\times 10^{5}$ & $9.7\times 10^{6}$ \\ 
                \cline{2-4}
                                        & $\alpha=2.5$ & $6.1\times 10^{6}$ & $1.2\times 10^{8}$\\
                \hline
                \multirow{2}{*}{$N=5,\,m=3$} & $\alpha=2$ & $1.4\times 10^{5}$ & $1.4\times 10^{6}$ \\ 
                \cline{2-4}
                                        & $\alpha=2.5$ & $1.7\times 10^{6}$ & $1.7\times 10^{7}$\\
			\hline
                \multicolumn{4}{p{240pt}}{The allowable error $\varepsilon=1\times10^{-4}$, and the batch size $S_t=T^\alpha$ where $T$ is set as 150.}
		\end{tabular} 
	\end{center}
\end{table}

\begin{table}[t]
	\caption{Oracle complexity required when varying $K$ (strongly mono.)}
	\label{tab.K}
	\begin{center}	
		\renewcommand\arraystretch{1.25}
		\begin{tabular}{|l|l|l|l|l|}
			\hline
			\multirow{2}{*}{Algorithm} & \multicolumn{3}{c|}{DVRSFBF} & \multirow{2}{*}{vr-SMFBS}\\
                \cline{2-4}
                          &  $K=10$  &  $K=20$  & $K=50$ &  \\
			\hline
			\makecell{Number of\\oracles} & $5.7\times 10^{6}$ & $6.6\times 10^{5}$ & $2.0\times 10^{6}$ & $1.4\times 10^{8}$\\ 
                \hline
                \multicolumn{5}{p{240pt}}{Size of network $N=20,\,m=7$. The allowable error $\varepsilon=1\times10^{-4}$. The batch size $S_t=\lfloor {\eta^{-2(t+1)}} \rfloor$, where $\eta=0.99$.}
		\end{tabular} 
	\end{center}
\end{table}

\section{Conclusion}
In this paper, we proposed a distributed variance-reduced stochastic forward-backward-forward splitting (DVRSFBF) algorithm for solving stochastic generalized Nash equilibrium problems (SGNEPs). Our approach incorporates variance reduction techniques into a structured monotone operator framework, allowing it to effectively handle problems with a possibly infinite sample space and a biased stochastic estimator. We provided a comprehensive theoretical analysis, proving that the proposed method achieves almost sure convergence under mild assumptions. Furthermore, we established a linear convergence rate when the mapping is strongly monotone, while the rate worsens to \( O(1/T) \) under monotonicity of the map. The variance reduction mechanism enables more efficient utilization of gradient information, leading to improved performance in equilibrium computation. Numerical experiments on a networked Cournot game demonstrate that DVRSFBF outperforms existing typical variance-reduced methods, exhibiting faster convergence and lower sampling cost. These results validate our theoretical findings and highlight the practical advantages of our approach in large-scale multi-agent systems. Future work includes exploring extensions to more general stochastic equilibrium problems beyond the monotone setting, as well as investigating distributed variance reduction schemes under constrained communication.


\section*{References}
\bibliographystyle{ieeetr}
\bibliography{mybib.bib}

\begin{thebibliography}{10}

\bibitem{ardagna2015generalized}
D.~Ardagna, M.~Ciavotta, and M.~Passacantando, ``Generalized {{N}ash}
  equilibria for the service provisioning problem in multi-cloud systems,''
  {\em IEEE Transactions on Services Computing}, vol.~10, no.~3, pp.~381--395,
  2015.

\bibitem{10237320}
K.~Lu, K.~Tang, S.~Dong, and Y.~Song, ``Generalized-{{N}ash}-equilibrium-based
  pareto solution for transmission-distribution-coupled optimal power flow,''
  {\em IEEE Transactions on Power Systems}, vol.~39, no.~2, pp.~4051--4063,
  2024.

\bibitem{FacKan07}
F.~Facchinei and C.~Kanzow, ``Generalized {N}ash equilibrium problems,'' {\em
  4OR}, vol.~5, no.~3, pp.~173--210, 2007.

\bibitem{facchinei2010}
F.~Facchinei and C.~Kanzow, ``Generalized {{N}ash} equilibrium problems,'' {\em
  Annals of Operations Research}, vol.~175, no.~1, pp.~177--211, 2007.

\bibitem{FacPalPanScu10}
G.~Scutari, D.~P. Palomar, F.~Facchinei, and J.-s. Pang, ``Convex optimization,
  game theory, and variational inequality theory,'' {\em IEEE Signal Processing
  Magazine}, vol.~27, no.~3, pp.~35--49 

\bibitem{facchinei2014vi}
F.~Facchinei, J.-S. Pang, G.~Scutari, and L.~Lampariello, ``{VI}-constrained
  hemivariational inequalities: distributed algorithms and power control in
  ad-hoc networks,'' {\em Mathematical Programming}, vol.~145, no.~1,
  pp.~59--96, 2014.

\bibitem{pavel2019}
L.~Pavel, ``Distributed {GNE} seeking under partial-decision information over
  networks via a doubly-augmented operator splitting approach,'' {\em IEEE
  Transactions on Automatic Control}, 2019.

\bibitem{BauCom16}
H.~H. Bauschke and P.~L. Combettes, {\em Convex Analysis and Monotone Operator
  Theory in Hilbert Spaces}.
\newblock Springer-CMS Books in Math., 2016.

\bibitem{fabiani2022stochastic}
F.~Fabiani and B.~Franci, ``A stochastic generalized {{N}ash} equilibrium model
  for platforms competition in the ride-hail market,'' in {\em 2022 IEEE 61st
  Conference on Decision and Control (CDC)}, pp.~4455--4460, IEEE, 2022.

\bibitem{MerStaIFAC19}
M.~Staudigl and P.~Mertikopoulos, ``Convergent noisy forward-backward-forward
  algorithms in non-monotone variational inequalities,'' {\em
  IFAC-PapersOnLine}, vol.~52, no.~3, pp.~120--125, 2019.

\bibitem{yu2017}
C.-K. Yu, M.~Van Der~Schaar, and A.~H. Sayed, ``Distributed learning for
  stochastic generalized {{N}ash} equilibrium problems,'' {\em IEEE
  Transactions on Signal Processing}, vol.~65, no.~15, pp.~3893--3908, 2017.

\bibitem{henrion2007}
R.~Henrion and W.~R{\"o}misch, ``On m-stationary points for a stochastic
  equilibrium problem under equilibrium constraints in electricity spot market
  modeling,'' {\em Applications of Mathematics}, vol.~52, no.~6, pp.~473--494,
  2007.

\bibitem{abada2013}
I.~Abada, S.~Gabriel, V.~Briat, and O.~Massol, ``A generalized
  {{N}ash}--{Cournot} model for the {Northwestern} {European} natural gas
  markets with a fuel substitution demand function: The {GaMMES} model,'' {\em
  Networks and Spatial Economics}, vol.~13, no.~1, pp.~1--42, 2013.

\bibitem{demiguel2009}
V.~DeMiguel and H.~Xu, ``A stochastic multiple-leader {Stackelberg} model:
  analysis, computation, and application,'' {\em Operations Research}, vol.~57,
  no.~5, pp.~1220--1235, 2009.

\bibitem{tang2007multicriterion}
Z.~Tang, J.-A. D{\'e}sid{\'e}ri, and J.~P{\'e}riaux, ``Multicriterion
  aerodynamic shape design optimization and inverse problems using control
  theory and {N}ash games,'' {\em Journal of Optimization Theory and
  Applications}, vol.~135, no.~3, pp.~599--622, 2007.

\bibitem{ramos2002pointwise}
A.~Ramos, R.~Glowinski, and J.~Periaux, ``Pointwise control of the burgers
  equation and related {N}ash equilibrium problems: computational approach,''
  {\em Journal of optimization theory and applications}, vol.~112,
  pp.~499--516, 2002.

\bibitem{gahururu2023risk}
D.~B. Gahururu, M.~Hinterm{\"u}ller, and T.~M. Surowiec, ``Risk-neutral
  pde-constrained generalized {N}ash equilibrium problems,'' {\em Mathematical
  Programming}, vol.~198, no.~2, pp.~1287--1337, 2023.

\bibitem{cui2023variance}
S.~Cui and U.~V. Shanbhag, ``Variance-reduced splitting schemes for monotone
  stochastic generalized equations,'' {\em IEEE Transactions on Automatic
  Control}, vol.~68, no.~11, pp.~6636--6648, 2023.

\bibitem{Bot2021}
R.~I. Boţ, P.~Mertikopoulos, M.~Staudigl, and P.~T. Vuong, ``Minibatch
  forward-backward-forward methods for solving stochastic variational
  inequalities,'' {\em Stochastic Systems}, vol.~11, no.~2, pp.~112--139, 2021.

\bibitem{YiPav19}
P.~Yi and L.~Pavel, ``An operator splitting approach for distributed
  generalized {{N}ash} equilibria computation,'' {\em Automatica}, vol.~102,
  pp.~111--121, 2019.

\bibitem{belgioioso2018}
G.~Belgioioso and S.~Grammatico, ``Projected-gradient algorithms for
  generalized equilibrium seeking in aggregative games are preconditioned
  forward-backward methods,'' in {\em 2018 European Control Conference (ECC)},
  pp.~2188--2193, IEEE, 2018.

\bibitem{kulkarni2012}
A.~A. Kulkarni and U.~V. Shanbhag, ``On the variational equilibrium as a
  refinement of the generalized {{N}ash} equilibrium,'' {\em Automatica},
  vol.~48, no.~1, pp.~45--55, 2012.

\bibitem{IusJofOliTho17}
A.~N. Iusem, A.~Jofr\'{e}, R.~I. Oliveira, and P.~Thompson, ``Extragradient
  method with variance reduction for stochastic variational inequalities,''
  {\em SIAM Journal on Optimization}, vol.~27, no.~2, pp.~686--724, 2017.

\bibitem{alacaoglu2022stochastic}
A.~Alacaoglu and Y.~Malitsky, ``Stochastic variance reduction for variational
  inequality methods,'' in {\em Conference on Learning Theory}, pp.~778--816,
  PMLR, 2022.

\bibitem{franci2021fbtac}
B.~Franci and S.~Grammatico, ``A distributed forward–backward algorithm for
  stochastic generalized {{N}ash} equilibrium seeking,'' {\em IEEE Transactions
  on Automatic Control}, vol.~66, no.~11, pp.~5467--5473, 2021.

\bibitem{JiaxXu09}
H.~Jiang and H.~Xu, ``Stochastic approximation approaches to the stochastic
  variational inequality problem,'' {\em IEEE Transactions on Automatic
  Control}, vol.~53, no.~6, pp.~1462--1475, 2008.

\bibitem{rosasco2016}
L.~Rosasco, S.~Villa, and B.~C. V{\~u}, ``Stochastic forward--backward
  splitting for monotone inclusions,'' {\em Journal of Optimization Theory and
  Applications}, vol.~169, no.~2, pp.~388--406, 2016.

\bibitem{franci2020fbecc}
B.~Franci and S.~Grammatico, ``A damped forward--backward algorithm for
  stochastic generalized {{N}ash} equilibrium seeking,'' in {\em 2020 European
  Control Conference (ECC)}, pp.~1117--1122, IEEE, 2020.

\bibitem{jofre2019variance}
A.~Jofr{\'e} and P.~Thompson, ``On variance reduction for stochastic smooth
  convex optimization with multiplicative noise,'' {\em Mathematical
  Programming}, vol.~174, no.~1-2, pp.~253--292, 2019.

\bibitem{jalilzadeh2018smoothed}
A.~Jalilzadeh, U.~V. Shanbhag, J.~Blanchet, and P.~W. Glynn, ``Smoothed
  variable sample-size accelerated proximal methods for nonsmooth stochastic
  convex programs,'' {\em Stochastic Systems}, vol.~12, no.~4, pp.~373--410,
  2022.

\bibitem{jalilzadeh20variable}
A.~Jalilzadeh, A.~Nedi{\'c}, U.~V. Shanbhag, and F.~Yousefian, ``A variable
  sample-size stochastic quasi-newton method for smooth and nonsmooth
  stochastic convex optimization,'' {\em Mathematics of Operations Research},
  vol.~47, no.~1, pp.~690--719, 2022.

\bibitem{johnson2013accelerating}
R.~Johnson and T.~Zhang, ``Accelerating stochastic gradient descent using
  predictive variance reduction,'' {\em Advances in neural information
  processing systems}, vol.~26, pp.~1--9, 2013.

\bibitem{nan2023extragradient}
T.~Nan, Y.~Gao, and C.~Kroer, ``Extragradient {SVRG} for variational
  inequalities: Error bounds and increasing iterate averaging,'' {\em arXiv
  preprint arXiv:2306.01796}, 2023.

\bibitem{NIPS2014_ede7e2b6}
A.~Defazio, F.~Bach, and S.~Lacoste-Julien, ``{SAGA}: A fast incremental
  gradient method with support for non-strongly convex composite objectives,''
  in {\em Advances in Neural Information Processing Systems}, vol.~27,
  pp.~1646--1654, Curran Associates, Inc., 2014.

\bibitem{palaniappan2016stochastic}
B.~Palaniappan and F.~Bach, ``Stochastic variance reduction methods for
  saddle-point problems,'' {\em Advances in Neural Information Processing
  Systems}, vol.~29, 2016.

\bibitem{zhang2023communication}
S.~Zhang, S.~Choudhury, S.~U. Stich, and N.~Loizou, ``Communication-efficient
  gradient descent-accent methods for distributed variational inequalities:
  Unified analysis and local updates,'' {\em arXiv preprint arXiv:2306.05100},
  2023.

\bibitem{alacaoglu2021forward}
A.~Alacaoglu, Y.~Malitsky, and V.~Cevher, ``Forward-reflected-backward method
  with variance reduction,'' {\em Computational optimization and applications},
  vol.~80, no.~2, pp.~321--346, 2021.

\bibitem{cai2022stochastic}
X.~Cai, C.~Song, C.~Guzm{\'a}n, and J.~Diakonikolas, ``Stochastic halpern
  iteration with variance reduction for stochastic monotone inclusions,'' {\em
  Advances in Neural Information Processing Systems}, vol.~35,
  pp.~24766--24779, 2022.

\bibitem{huang2022accelerated}
K.~Huang, N.~Wang, and S.~Zhang, ``An accelerated variance reduced extra-point
  approach to finite-sum vi and optimization,'' {\em arXiv preprint
  arXiv:2211.03269}, 2022.

\bibitem{escobar2005oligopolistic}
J.~F. Escobar and A.~Jofre, ``Oligopolistic competition in electricity spot
  markets,'' {\em Working paper, SSRN}, 2005.

\bibitem{hu2004electricity}
X.~Hu, D.~Ralph, E.~K. Ralph, P.~Bardsley, M.~C. Ferris, {\em et~al.},
  ``Electricity generation with looped transmission networks: Bidding to an
  iso,'' {\em Research Paper}, vol.~16, 2004.

\bibitem{Kannan2013}
A.~Kannan, U.~V. Shanbhag, and H.~M. Kim, ``Addressing supply-side risk in
  uncertain power markets: stochastic {N}ash models, scalable algorithms and
  error analysis,'' {\em Optimization Methods and Software}, vol.~28, no.~5,
  pp.~1095--1138, 2013.

\bibitem{ZHONG2022}
Y.~Zhong, T.~Yang, B.~Cao, and T.~Cheng, ``On-demand ride-hailing platforms in
  competition with the taxi industry: Pricing strategies and government
  supervision,'' {\em International Journal of Production Economics}, vol.~243,
  p.~108301, 2022.

\bibitem{offplatform23}
E.~J. He, S.~Savin, J.~Goh, and C.-P. Teo, ``Off-platform threats in on-demand
  services,'' {\em Manufacturing \& Service Operations Management}, vol.~25,
  no.~2, pp.~775--791, 2023.

\bibitem{ridesharing19}
K.~Bimpikis, O.~Candogan, and D.~Saban, ``Spatial pricing in ride-sharing
  networks,'' {\em Operations Research}, vol.~67, no.~3, pp.~744--769, 2019.

\bibitem{ride-hail22}
F.~Fabiani and B.~Franci, ``A stochastic generalized {N}ash equilibrium model
  for platforms competition in the ride-hail market,'' in {\em 2022 IEEE 61st
  Conference on Decision and Control (CDC)}, pp.~4455--4460, 2022.

\bibitem{facchineikanzow2007}
F.~Facchinei and C.~Kanzow, ``Generalized {N}ash equilibrium problems,'' {\em
  4or}, vol.~5, no.~3, pp.~173--210, 2007.

\bibitem{FacFisPic07}
F.~Facchinei, A.~Fischer, and V.~Piccialli, ``On generalized {N}ash games and
  variational inequalities,'' {\em Operations Research Letters}, vol.~35,
  no.~2, pp.~159--164, 2007.

\bibitem{belgioioso2023}
G.~Belgioioso and S.~Grammatico, ``Semi-decentralized generalized {{N}ash}
  equilibrium seeking in monotone aggregative games,'' {\em IEEE Transactions
  on Automatic Control}, vol.~68, no.~1, pp.~140--155, 2023.

\bibitem{franci2022partial}
B.~Franci and S.~Grammatico, ``Stochastic generalized {N}ash equilibrium
  seeking under partial-decision information,'' {\em Automatica}, vol.~137,
  p.~110101, 2022.

\bibitem{facchinei2007vi}
F.~Facchinei, A.~Fischer, and V.~Piccialli, ``On generalized {{N}ash} games and
  variational inequalities,'' {\em Operations Research Letters}, vol.~35,
  no.~2, pp.~159--164, 2007.

\bibitem{FacPan03}
F.~Facchinei and J.-S. Pang, {\em Finite-Dimensional Variational Inequalities
  and Complementarity Problems - Volume I and Volume II}.
\newblock Springer Series in Operations Research, 2003.

\bibitem{auslender2000}
A.~Auslender and M.~Teboulle, ``Lagrangian duality and related multiplier
  methods for variational inequality problems,'' {\em SIAM Journal on
  Optimization}, vol.~10, no.~4, pp.~1097--1115, 2000.

\bibitem{godsil2013}
C.~Godsil and G.~F. Royle, {\em Algebraic graph theory}, vol.~207.
\newblock Springer Science \& Business Media, 2013.

\bibitem{cui2021relaxed}
S.~Cui, B.~Franci, S.~Grammatico, U.~V. Shanbhag, and M.~Staudigl, ``A
  relaxed-inertial forward-backward-forward algorithm for stochastic
  generalized {N}ash equilibrium seeking,'' in {\em 2021 60th IEEE Conference
  on Decision and Control (CDC)}, pp.~197--202, IEEE, 2021.

\bibitem{franci2020fbf}
B.~Franci, M.~Staudigl, and S.~Grammatico, ``Distributed forward-backward
  (half) forward algorithms for generalized {{N}ash} equilibrium seeking,'' in
  {\em 2020 European Control Conference (ECC)}, pp.~1274--1279, 2020.

\bibitem{grammatico2018}
S.~Grammatico, ``Comments on ``distributed robust adaptive equilibrium
  computation for generalized convex games''[automatica 63 (2016) 82--91],''
  {\em Automatica}, vol.~97, pp.~186--188, 2018.

\bibitem{Tse00}
P.~Tseng, ``A modified forward-backward splitting method for maximal monotone
  mappings,'' {\em SIAM Journal on Control and Optimization}, no.~2,
  pp.~431--446.

\bibitem{ByrChiNocWu12}
R.~H. Byrd, G.~M. Chin, J.~Nocedal, and Y.~Wu, ``Sample size selection in
  optimization methods for machine learning,'' {\em Mathematical Programming},
  vol.~134, no.~1, pp.~127--155, 2012.

\bibitem{lei2022distributed}
J.~Lei and U.~V. Shanbhag, ``Distributed variable sample-size gradient-response
  and best-response schemes for stochastic {{N}ash} equilibrium problems,''
  {\em SIAM Journal on Optimization}, vol.~32, no.~2, pp.~573--603, 2022.

\bibitem{Pol87}
B.~T. Polyak, {\em Introduction to Optimization}.
\newblock Optimization Soft., 1987.

\bibitem{ahmadi2016analysis}
H.~Ahmadi, {\em On the analysis of data-driven and distributed algorithms for
  convex optimization problems}.
\newblock The Pennsylvania State University, 2016.

\bibitem{Nes07}
Y.~Nesterov, ``Dual extrapolation and its applications to solving variational
  inequalities and related problems,'' {\em Mathematical Programming},
  vol.~109, no.~2, pp.~319--344, 2007.

\end{thebibliography}

\end{document}